\documentclass{amsart}

\usepackage{amsmath}
\usepackage{amsthm}
\usepackage{amsfonts}
\usepackage{amssymb,enumerate,verbatim,xspace}
\numberwithin{equation}{section}
\usepackage[cyr]{aeguill}

\renewcommand{\a}{{\overline{a}}}
\renewcommand{\b}{{\overline{b}}}
\newcommand{\x}{{\overline{x}}}
\newcommand{\y}{{\overline{y}}}

\newcommand{\z}{{\overline{z}}}

\newcommand{\stab}{\operatorname{stab}}
\newcommand{\lgth}{\operatorname{lg}}

\newcommand{\tp}{\operatorname{tp}}

\newcommand{\dcl}{\operatorname{dcl}}
\newcommand{\acl}{\operatorname{acl}}

\newcommand{\Aut}{\operatorname{Aut}}

\newcommand{\cl}{\operatorname{cl}}
\newcommand{\SU}{\operatorname{SU}}

\renewcommand{\d}{\operatorname{d}}

\newcommand{\Gm}{\mathbb{G}_m}

\newcommand{\Loc}{\operatorname{locus}}
\newcommand{\Gal}{\operatorname{Gal}}
\newcommand{\DM}{\operatorname{MD}}
\newcommand{\Div}{\operatorname{Div}}
\newcommand{\RM}{\operatorname{MR}}
\newcommand{\RDM}{\operatorname{MRD}}

\newcommand{\Th}{\operatorname{Th}}

\newcommand{\Cb}{\operatorname{Cb}}
\newcommand{\cd}{\operatorname{cd}}

\renewcommand{\L}{\mathcal L}
\newcommand{\C}{\mathcal C_0}
\newcommand{\Cmu}{\mathcal{C}_0^{\mu}}

\newcommand{\RegP}{\mathfrak{Reg}}

\newcommand{\R}{\mathbb{R}}
\newcommand{\N}{\mathbb{N}}
\newcommand{\Z}{\mathbb{Z}}
\newcommand{\F}{\mathbb{F}}
\newcommand{\Q}{\mathbb{Q}}

\newcommand{\Cn}[1]{(\mathbb{C}^*)^{#1}}

\newcommand{\res}{\!\upharpoonright\!} 
\renewcommand{\phi}{\varphi}
\newcommand{\elex}{\preccurlyeq}
\newcommand{\elres}{\succcurlyeq}
\newcommand{\g}{\mathfrak g}

\newcommand{\UE}{\textrm{\"U}}

\newcommand{\ldim}{\operatorname{l.dim}_\Q}
\newcommand{\tr}{\operatorname{tr.deg}}

\newcommand{\divh}{\operatorname{div}}

\def\Ind#1#2{#1\setbox0=\hbox{$#1x$}\kern\wd0\hbox to 0pt{\hss$#1\mid$\hss}
\lower.9\ht0\hbox to 0pt{\hss$#1\smile$\hss}\kern\wd0}
\def\ind{\mathop{\mathpalette\Ind{}}}
\def\Notind#1#2{#1\setbox0=\hbox{$#1x$}\kern\wd0\hbox to 0pt{\mathchardef
\nn="3236\hss$#1\nn$\kern1.4\wd0\hss}\hbox to 
0pt{\hss$#1\mid$\hss}\lower.9\ht0
\hbox to 0pt{\hss$#1\smile$\hss}\kern\wd0}
\def\nind{\mathop{\mathpalette\Notind{}}}

\theoremstyle{plain}
\newtheorem{Thm}{Theorem}[section]
\newtheorem{Prop}[Thm]{Proposition}
\newtheorem{Cor}[Thm]{Corollary}
\newtheorem{Lem}[Thm]{Lemma}

\newtheorem{Rem}[Thm]{Remark}
\newtheorem{Fact}[Thm]{Fact}

\theoremstyle{definition}
\newtheorem{Defi}[Thm]{Definition}

\newtheorem{Pb}[Thm]{Question}

\newtheorem{Exs}[Thm]{Examples}
\newtheorem{Ex}[Thm]{Example}

\begin{document}
\title{Generic Automorphisms and Green Fields}
\author{Martin Hils}
\date{\today}

\address{Institut de Math\'ematiques de Jussieu\\
\'Equipe de Logique Math\'ematique\\
Universit\'e Paris Diderot Paris 7\\
UFR de Math\'ematiques - case 7012, site Chevaleret\\
75205 Paris Cedex 13\\
France} 
\email{hils@logique.jussieu.fr}

\begin{abstract}
We show that the generic automorphism is axiomatisable in the green field of Poizat (once Morleyised) as well as 
in the bad fields which are obtained by collapsing this green field to finite Morley rank. As a corollary, we obtain 
``bad pseudofinite fields" in characteristic 0.

In both cases, we give geometric axioms. In fact, a general framework is presented allowing this kind 
of axiomatisation. We deduce from various constructibility results for algebraic varieties in characteristic 0 
that the green and bad fields fall into this framework. Finally, we give similar results for other theories 
obtained by Hrushovski amalgamation, e.g.\ the free fusion of two strongly minimal theories having the 
definable multiplicity property.

We also close a gap in the construction of the bad field, showing that the codes may be chosen to be families of 
strongly minimal sets.
\end{abstract}

\maketitle

\section{Introduction}
For more than two decades now, new and often unexpected stable structures have been constructed using 
Hrushovski's amalgamation method, starting in 1988 when Hrushovski obtained a strongly minimal 
theory which violated Zilber's trichotomy conjeture (see \cite{Hr93}). This construction is called the \emph{ab initio} 
case. The \emph{fusion} of two strongly minimal structures having DMP (i.e.\ definable Morley degrees) into a single 
one \cite{Hr92} then showed that the realm of strongly minimal theories is vast, even when one only looks at strongly minimal expansions of algebraically closed fields. 

Poizat's \emph{bicoloured fields} are expansions of algebraically closed fields by a new predicate. The black fields (where a new subset is added) answer a question of Berline and Lascar 
about possible ranks of superstable fields \cite{Po99}. The construction of the green fields, algebraically closed fields of characteristic 0 with a proper subgroup of the multiplicative group of the field, requires non-trivial results from algebraic 
geometry, in order to establish the relevant definability properties needed for the amalgamation construction to work 
\cite{Po01}. Poizat's green fields are infinite rank analogues of so-called \emph{bad fields}, fields of finite Morley rank 
with a definable proper infinite subgroup of the multiplicative group. In positive characteristic, bad fields are very unlikely to exist, 
by a result of Wagner \cite{Wa03}. Their absence would have simplified the study of groups of finite Morley rank, in particular 
that of infinite simple groups of finite Morley rank which according to Cherlin-Zilber's Algebraicity Conjecture should 
be algebraic groups. Baudisch, Martin-Pizarro, Wagner and the author showed in \cite{BHMW07} that Poizat's green field may 
be collapsed into a bad field. 

The positivity of the predimension is one of the key features of Hrushovski's amalgamation method. Zilber suggested 
that one interprets this as a \emph{generalised Schanuel condition}, due to the analogy with Schanuel's Conjecture 
(SC) which asserts that for any $\Q$-linearly independent tuple $(y_1,\ldots,y_n)$ of complex numbers one 
has $\tr(y_1,\ldots,y_n,e^{y_1},\ldots,e^{y_n}/\Q)\geq n$. This conjecture is wide open. Ax proved a 
differential version of it \cite{Ax71}. In the case of the green fields, the analogy raised by Zilber is supported by two facts. 
On the one hand, assuming (SC), Zilber constructs a natural model of the theory of the green field of Poizat, with universe 
the complex numbers and which has an ``analytic" flavour. On the other hand, Ax's result, or rather a consequence thereof called \emph{weak CIT}, is 
essentially used in the construction by Poizat. Weak CIT is a finiteness result on intersections of algebraic varieties with cosets 
of tori in characteristic 0 which allows to control atypical components of such intersections, i.e.\ those having a greater dimension than the expected one.

If $T$ is a stable model-complete theory, one may build the theory $T_{\sigma}$ of models of $T$ with a 
distinguished automorphism. It is an interesting question to determine whether $T_{\sigma}$ admits a model-companion. If 
it does, we denote it by $TA$ and say that the generic automorphism is axiomatisable in $T$. 
The geometric model theory of $TA$ (paradigmatically that of $ACFA$ in \cite{CH99}) has proven to be 
a powerful tool when applied to problems in algebraic geometry, number theory and algebraic dynamics (see e.g.\ \cite{Hr01,Sc02,Hr04,CH07,CH07a}). Whether $TA$ exists or not 
is a test question on how well one definably controls `multiplicities' in $T$. Existence of $ACFA$ for example is an easy 
consequence of the fact that being irreducible is definable in families of algebraic varietes; the abstract analogue 
of this for a theory of finite Morley rank is the definable multiplicity property (DMP).

For many structures obtained by Hrushovski amalgamation (when definably expanded to a language in which they become model-complete), it is quite elementary to show that the generic 
automorphism is axiomatisable, using so-called `geometric axioms'. However, in the green fields of Poizat and in the bad 
fields, using just weak CIT one only gains good definable control of dimension and rank. In this vein, there is 
the result of Evans that the green fields do not have the finite cover property \cite{Ev08}. But in order to axiomatise the 
generic automorphism, we also need a definable control of `multiplicities'. There are difficulties related to a necessary choice 
of green roots, and Kummer theory comes into the picture.

\smallskip

The paper is organised as follows. In Section \ref{S:GenAuto} we present a framework for `geometric 
axioms' for $TA$ in the case where $T$ is a stable, complete and model-complete theory: we show in Proposition 
\ref{P:PartialGenTAExists} that such an axiomatisation may be given if $T$ admits a \emph{geometric notion of 
genericity} (see Definition \ref{D:GeometricGenericity}). We then review the construction of the green fields of Poizat and of the bad fields, including 
the relevant uniformity results from algebraic geometry used in the course of the construction (Section \ref{S:delta}). 

Section \ref{S:def} is devoted to the proof of a definability result in characteristic 0. We show 
that being Kummer generic is definable for algebraic varieties $V$ among an algebraic family 
(Proposition \ref{P:SimpleDef}), where Kummer genericity of $V$ is a property defined in terms of Kummer extensions 
of the field of rational functions $K(V)$. Definability of Kummer genericity is then used to overcome the difficulties related to the 
choice of green roots which were mentioned above. This enables us to close a gap in the construction of the bad field which had been observed 
by Roche (see Corollary \ref{C:SmCodes}).

In Section \ref{S:Gen_Green}, we use the definability of Kummer genericity  to 
prove the main results of the paper, namely that the generic automorphism 
is axiomatisable in the green field of Poizat (Theorem \ref{T:GenGreen}) and in the bad fields (Theorem 
\ref{T:Gen_Bad_Field}). From the latter result, passing to the fixed structure, we deduce the existence of `bad pseudofinite fields' (Corollary \ref{C:Psf_Bad}). 

Finally, in the last section, we mention existence results of $TA$ for other theories obtained by 
Hrushovski amalgamation. The common feature is that a geometric notion of genericity may be exhibited in these 
theories in a straight-forward way, using the respective `base theories'. A full proof is presented in the case of the 
free fusion of two strongly minimal theories having DMP. The section also includes a brief review of Hrushovski's 
amalgamation method .

I would like to thank Zo\'e Chatzidakis, Frank Wagner and Boris Zilber for helpful discussions on the subject, and the anonymous referee for some useful comments.

\section{Generic Automorphisms of Stable Theories}\label{S:GenAuto}

Let $T$ be a complete $\L$-theory, and let $\L_{\sigma}=\L\cup\{\sigma\}$, 
where $\sigma$ is a new unary funcion symbol. We consider the $\L_{\sigma}$-theory $T_{\sigma}$ obtained 
by adding to $T$ axioms expressing that $\sigma$ is an $\L$-automorphism of the corresponding 
model of $T$. If $T$ is model-complete, it follows that $T_{\sigma}$ is an 
inductive theory, so it has a model-companion (which we denote by $TA$ if it exists) if and only if 
the class of its existentially closed models is elementary. In this case, we say that \emph{the 
generic automorphism is axiomatisable in $T$}, or that \emph{$TA$ exists}.

If $T$ is an arbitrary complete theory, we say that the generic automorphism is axiomatisable in $T$ 
if this holds for some expansion by definitions $T^*$ of $T$ which is model-complete. This does not 
depend on the choice of $T^*$, and so we may as well assume that $T^*$ eliminates quantifiers, by 
taking the Morleyisation of $T$. Hence we really deal with some kind of relative existence of a model-companion. 

If $TA$ exists for some stable theory $T$, then all its completions are simple (see Fact \ref{T:CP} below), and in 
general unstable. The reader may consult \cite{Wa00} 
for a survey on simple theories; although, we will make no real use of them in the present paper.

Some notation: in any model $(M,\sigma)$ of $T_{\sigma}$ we denote 
by $\acl(A)$ the algebraic closure of $A$ in the sense of $M^{eq}\models T^{eq}$, and by $\acl_{\sigma}(A)$ 
the set $\acl\left(\bigcup_{z\in\Z}\sigma^z(A)\right)$, a subset of $M^{eq}$ which is easily 
seen to be closed under (the induced actions on $M^{eq}$ of) $\sigma$ and $\sigma^{-1}$, and algebraically closed in the sense of $T^{eq}$. 

\subsection{Some known results} 
If $T$ has the strict order property, then $TA$ does not exist \cite{KS02}. Kikyo and Pillay conjectured that the 
existence of $TA$ implies that $T$ is stable \cite{KP00}. 
In the following, we will concentrate on stable 
theories. For the rest of this section, \emph{we will assume that $T$ is complete, model-complete and stable}. Fact \ref{T:CP} lists some basic results shown by Chatzidakis and Pillay \cite{CP98}.

\begin{Fact}\label{T:CP}
Let $T$ be a stable complete theory with quantifier elimination such that $TA$ exists. 

\begin{enumerate}
\item The algebraic closure in models of $TA$ is given by $\acl_{\sigma}$. 
\item For $A_i\subseteq M_i\models TA$, $i=1,2$ one has $A_1\equiv_{\L_{\sigma}}A_2$ if and only if 
$\acl_{\sigma}(A_1)\cong_{\mathcal{L}_{\sigma}}\acl_{\sigma}(A_2)$ (over the map sending 
$A_1$ to $A_2$). In particular, the completions of $TA$ are given by the $\mathcal{L}_{\sigma}$-isomorphism types of $\acl(\emptyset)=\acl_{\sigma}(\emptyset)$.   
\item Any completion $\tilde{T}$ of $TA$ is simple (supersimple if $T$ is superstable), and the following characterisation of non-forking holds:

$$A\ind_B^{\tilde{T}} C\,\,\,\,\Leftrightarrow
\,\,\,\,\acl_{\sigma}(AB)\ind^T_{\acl_{\sigma}(B)}\acl_{\sigma}(BC).$$
\item Assume in addition that $T$ eliminates imaginaries and that any algebraically closed set is a model of $T$. Then any completion of $TA$ eliminates 
imaginaries, and the definable set 
$F=Fix(\sigma)=\{m\in M\,  \mid\,   \sigma(m)=m\}$ is stably embedded in $M$.
\end{enumerate}
\end{Fact}

The existence of $TA$ may be considered as a (very nice) property of the initial theory $T$. 
Kudaiberganov observed that for a stable theory $T$ it implies 
$T$ does not have the finite cover property (i.e.\ is nfcp). Baldwin and Shelah \cite{BS01} gave 
an abstract characterisation of those stable theories $T$ for which $TA$ exists. (It consists of a strengthening of nfcp, is purely in terms of 
$T$ and uses $\Delta$-types.) In this direction, one may also mention 
the following result due to Hasson and Hrushovski. 

\begin{Fact}[\cite{HH07}]Let $T$ be a strongly minimal theory. Then $TA$ exists if and only if $T$ has the DMP.
\end{Fact}

Recall that a theory $T$ of finite Morley rank has the \emph{DMP (definable multiplicity property)} if 
for any pair of natural numbers $(r,d)$ and any formula $\phi(\x,\b)$ with $\RDM(\phi(\x,\b))=(r,d)$ there 
exists $\theta(\z)\in\tp(\b)$ such that $\RDM(\phi(\x,\b'))=(r,d)$ whenever $\models\theta(\b')$. 
(See \cite{Hr92} for a discussion of the DMP.)

\subsection{A framework for geometric axioms} \label{S:GeometricFranework}
The framework we present here allows to unify existing proofs showing that $TA$ exists 
for specific stable theories $T$. The common feature of these proofs is the axiomatisation 
of $TA$ in terms of what is called `geometric axioms'. In a way, we give in 
the sequel a general principle to organise such proofs. Compared to the characterisation of stable 
complete theories in which the generic automorphism is axiomatisable given in  \cite{BS01}, the 
criterion we present is of a more `geometric' nature, since it brings global considerations into play. 

Before we give definitions, let us start with a motivating example. If $(M,\sigma)\models T_{\sigma}$ and $X\subseteq M^n$ is $\L_M$-definable, say 
$X=\phi(M,\b)$ for some $\L$-formula $\phi(\x,\y)$ and $\b\in M$, let $X^{\sigma}=\{\sigma(\overline{c})\in M^n\mid\overline{c}\in X\}=\phi(M,\sigma(\b))$. 
Clearly, $\RDM(X)=\RDM(X^{\sigma})$.

\begin{Ex}[\cite{CP98}]\label{Ex:AddDMP}
Let $T$ be a theory of finite Morley rank with DMP and such that $\RM$ is additive: for all $\a,\b$ 
and $C$ one has $\RM(\a\b/C)=\RM(\a/\b C)+\RM(\b/C)$. 

Then $TA$ exists. More precisely, if $(M,\sigma)\models T_{\sigma}$, then $(M,\sigma)$ is existentially closed (i.e. 
a model of $TA$) if and 
only if the following condition holds:

\begin{itemize}\item[($*$)]Assume that $X\subseteq M^n$ and $Y\subseteq X\times X^{\sigma}$ are $\L_M$-definable sets of Morley degree 1, such 
that if $(\a,\a')$ is generic in $Y$ (over $M$), then $\a$ is generic in $X$ and $\a'$ is generic in $X^{\sigma}$. Then, there exists $\overline{c}\in M^n$ such 
that $(\overline{c},\sigma(\overline{c}))\in Y$.
\end{itemize} 

\smallskip

If $f:Y\rightarrow X$ is a definable function, with $\RDM(Y)=(n,1)$, $\RDM(X)=(m,1)$, then, by the additivity of $\RM$, $f$ maps the generic type of $Y$ to the generic type of 
$X$ if and only if $\RM(\{\a\in X\,\mid\,\RM(f^{-1}(\a))=n-m\})=m$. Since $\RM$ and $\DM$ are definable in $T$ by assumption, 
this shows that the condition ($*$) may be expressed in a first order way. 
\end{Ex}

Let  $R_{\g}$ be a relation defined on pairs of the form $(p(\x),\phi(\x))$, where $p(\x)\in S(M)$ is a finitary type over some 
model $M\models T$ and $\phi(\x,\b)\in p$ is a formula. 
\begin{itemize}
\item If $(p,\phi)$ is in $R_{\g}$, we say that \emph{$p$ is generic in $\phi$}. A tuple $\a\in N\elres M$ is 
\emph{generic in $\phi$ over $M$} (where $\phi$ is a formula with parameters from the model $M$) 
if $p:=\tp(\a/M)$ is generic in $\phi$, i.e.\ if the pair $(p,\phi)$ is in $R_{\g}$. 
\item A formula $\psi(\x,\z)$ (without parameters) is called \emph{nice} if for any $\b$ with 
$\psi(\x,\b)\neq\emptyset$ and any model $M$ containing $\b$ there is a unique type $p\in S(M)$ 
which is generic in $\psi(\x,\b)$. A formula $\psi(\x,\b)\neq\emptyset$ is called nice if $\psi(\x,\z)$ is 
nice.
\item A type $p\in S(M)$ is \emph{nice} if $N(p)\vdash p$ holds, where
$$N(p):=\{\psi(\x,\b)\mbox{ nice }\,\mid\, p 
\mbox{ is generic in } \psi(\x,\b)\}.$$    
\end{itemize}

\begin{Defi} \label{D:GeometricGenericity}
The relation $R_{\g}$ is a \emph{geometric notion of genericity} (for $T$) if the following properties 
hold:

\begin{enumerate}
\item[(1)] (\emph{Invariance.}) $R_{\g}$ is invariant under automorphisms. 
\item[(2)] (\emph{Coherence.}) Let $p\in S(N)$, $M\elex N$ and $\phi$ a nice formula with parameters 
from $M$. Then the (unique) generic type of $\phi$ over $N$ restricts to the generic type of $\phi$ over $M$. 
\item[(3)] (\emph{Enough nice types.}) For every $n$-type $p_0$ over some model $M$, there is 
(for some $m$) a nice type $p\in S_{n+m}(M)$ such that $\pi(p)=p_0$, where $\pi$ 
is the natural projection $S_{n+m}(M)\rightarrow S_n(M)$. 
\item[(4)] (\emph{Definability of generic projections.}) Let $\x\supseteq \x_1$, and let $\psi(\x,\b)$ and 
$\phi(\x_1,\b_1)$ be nice formulas, with generic types $p(\x)$ and 
$p_1(\x_1)$ in $S(M)$, respectively. Assume that $\models\psi(\x,\b)\rightarrow\phi(\x_1,\b_1)$ and 
$\pi_1(p)=p_1$. Then there is $\theta(\z,\z_1)\in\tp(\b\b_1)$ such that 
for all $\b'\b'_1\models\theta(\z,\z_1)$, one has $\psi(\x,\b')\neq\emptyset$, $\models\psi(\x,\b')\rightarrow\phi(\x_1,\b'_1)$ and the unique generic type $p'$ of $\psi(\x,\b')$ projects onto the unique generic type 
$p_1'$ of $\phi(\x_1,\b_1')$.
\end{enumerate}
\end{Defi}

In a stable theory the non-forking extension of a stationary type is the unique extension which is 
invariant under automorphisms. Thus, for a nice formula $\phi(\x,\b)$ and $\b\in M\elex N$, the generic 
type of $\phi$ over $N$ is the non-forking extension of the generic type over $M$.

Here is the result the notion is made for.

\begin{Prop}     \label{P:PartialGenTAExists}
Suppose $T$ admits a geometric notion of genericity $R_{\g}$. Then $TA$ exists. 
\end{Prop}

\begin{proof}

We give `geometric axioms' for $TA$, using the relation $R_{\g}$. 
Let $(M,\sigma)\models T_{\sigma}$ and $\tilde{p}(\x,\x',\x_r)\in S_{2n+k}(M)$ a nice type restricting to 
nice types $p(\x)$ and $p'(\x')$ in $S_n(M)$ such that $p'=\sigma(p)$. 
Let $\psi(\x,\x',\x_r,\tilde{b})\in N(\tilde{p})$, $\phi(\x,\b)\in N(p)$ and thus (by invariance) 
$\phi(\x',\sigma(\b))\in N(p')$ such that 
$$\models\psi(\x,\x',\x_r,\tilde{b})\rightarrow\phi(\x,\b)\land\phi(\x',\sigma(\b)).$$
Moreover, let $\theta(\tilde{z},\z)$ and $\theta'(\tilde{z},\z')$ be given by property (4) 
applied to the pairs of formulas $(\psi(\x,\x',\x_r,\tilde{b}),\phi(\x,\b))$ 
and $(\psi(\x,\x',\x_r,\tilde{b}),\phi(\x',\sigma(\b)))$. Put
$\Theta(\tilde{z},\z):=\theta(\tilde{z},\z)\wedge\theta'(\tilde{z},\sigma(\z))$. The corresponding 
axiom for this choice of formulas is: 
$$\forall\tilde{z}\z\exists\x\,\x_r[\Theta(\tilde{z},\z)
\rightarrow\psi(\x,\sigma(\x),\x_r,\tilde{z})].$$
We call these axioms the \emph{geometric axioms}. We will show that a model of $T_{\sigma}$ is existentially closed if and only if it satisfies the 
geometric axioms. This is straight-forward and will finish the proof.

Let $(M,\sigma)$ be an e.c.\ model of $T_{\sigma}$, and suppose that $(M,\sigma)\models\Theta(\tilde{b},\b)$. This means that the unique generic 
type $\tilde{p}(\x,\x',\x_r)$ of 
$\psi(\x,\x',\x_r,\tilde{b})$ (over $M$) restricts to the unique 
generic types $p$ and $p'=p^{\sigma}$ of $\phi(\x,\b)$ and 
$\phi(\x',\sigma(\b))$, respectively ($p'=\sigma(p)$ being a consequence of invariance). Choose 
$\tilde{a}=(\a,\a',\a_r)\models\tilde{p}$ (with $\tilde{a}$ from some $M^*\elres_{\L}M$). 
So $\a\models p$ and $\a'\models p^{\sigma}$. Going to some extension of $M^*$ if necessary, 
we may thus assume there is $\sigma^*\in Aut(M^*)$ extending 
$\sigma$ such that $\sigma^*(\a)=\a'$. In particular, $(M^*,\sigma^*)\models 
\exists\x\,\x_r\psi(\x,\sigma(\x),\x_r,\tilde{b})$. So the same is true in $(M,\sigma)$, as this is an e.c.\ 
model, and $(M,\sigma)$ satisfies the geometric axiom corresponding to $\psi$ and $\phi$. 

Conversely, let $(M,\sigma)$ be a model of $T_{\sigma}$ together with all geometric axioms. Let 
$(M,\sigma)\subseteq (N,\sigma)\models T_{\sigma}$ and $\a_0\in N$, satisfying some 
quantifier free $\L_{\sigma}$-formula with parameters from $M$. A standard reduction shows 
that we may assume that this formula is of the form $\chi(\x_0,\sigma(\x_0),\b_0)$, where 
$\chi(\x_0,\x_0',\z_0)$ is an $\L$-formula (without quantifiers). Put $p_0:=\tp_{\L}(\a_0/M)$ and 
$p_0'=p_0^{\sigma}:=\tp_{\L}(\sigma(\a_0)/M)\in S_n(M)$. By condition (3) in Definition \ref{D:GeometricGenericity}, there is a nice type $p\in S_{n+m}(M)$ restricting to $p_0$. 
Replacing $(N,\sigma)$ by some extension $(\tilde{N},\tilde{\sigma})$ if necessary, we may assume 
there exists a tuple $\a\in N$ containg $\a_0$ such that $\models p(\a)$. So $\sigma(\a)\supseteq\sigma(\a_0)$. Again by (3), applied to $\tp_{\L}(\a,\sigma(\a)/M)$, we may choose a nice type 
$\tilde{p}(\x,\x',\x_r)\in S(M)$ restricting to $p(\x)$ and 
$p'=p^{\sigma}$, respectively.

Now we choose some arbitrary $\phi(\x,\b)\in N(p)$, then we choose 
$\psi(\x,\x',\x_r,\tilde{b})\in N(\tilde{p})$ 
such that $\psi(\x,\x',\x_r,\tilde{b})$ implies $\chi(\x_0,\x_0',\b_0)\wedge\phi(\x,\b)\wedge\phi(\x',\sigma(\b))$ (this is possible since $N(\tilde{p})\vdash\tilde{p}$).  

Since $(M,\sigma)\models\Theta_{\psi,\phi}(\tilde{b},\b)$, the corresponding axiom ensures that there 
are tuples $\overline{\alpha},\overline{\alpha}_r\in M$ such that  
$M\models \psi(\overline{\alpha},\sigma(\overline{\alpha}),
\overline{\alpha}_r,\tilde{b})$. In particular, $(M,\sigma)\models 
\chi(\overline{\alpha}_0,\sigma(\overline{\alpha}_0),\b_0)$, where $\overline{\alpha}_0$ denotes 
the appropriate subtuple of $\overline{\alpha}$. This shows that $(M,\sigma)$ is an e.c.\ model. 
\end{proof}

\begin{Exs}\label{R:GeometricExamples}
The following known proofs of existence of $TA$ are instances of Proposition 
\ref{P:PartialGenTAExists}. (1) and (2) are from \cite{CP98}, and (3) is in \cite{Bu06}.

\smallskip
\noindent
(1) Let $T$ be a theory of finite additive Morley rank with DMP (see Example \ref{Ex:AddDMP}).

Genericity with repect to $\RM$ gives rise to a geometric notion of genericity. Nice formulas 
correspond to formulas with all instances of degree 1. Properties (1) and (2) from Definition \ref{D:GeometricGenericity} are easily verified, (4) 
follows from additivity of $\RM$ combined with the DMP as indicated in Example \ref{Ex:AddDMP}, and (3) is a consequence of the DMP (this is a degenerate case since all types over models are nice).  
\smallskip

\noindent
(2) Let $T$ be the theory of a totally transcendental module, or more generally a complete 
theory of a one-based group $G$ which is totally transcendental. Any definable subset of $G^n$ is given by a boolean 
combination of cosets of $\acl^{eq}(\emptyset)$-definable 
(connected) subgroups of $G^n$ (see e.g.\ \cite[Cor. 4.4.6]{Pi96}); any strong type $p$ is the (unique) generic type of a coset of its stabiliser $\stab(p)$.

It is straightforward to check that genericity with respect to Morley rank (or with respect to stable group 
theory, this amounts to the same in this context) gives rise to a geometric notion of genericity. Nice formulas are formulas with all instances of Morley degree 1.

\smallskip
\noindent
(3) Let $DCF_0$ be the theory of differentially closed fields in characteristic 0. It is 
shown in \cite[Corollary 2.15]{Bu06} (rephrased in our terminology) that \emph{$D$-genericity} with 
respect to the Kolchin topology is a geometric notion of genericity.
\end{Exs}

\section{Green fields}\label{S:delta}
We present in this section a sketch of Poizat's construction of a \emph{green field} in characteristic 0 \cite{Po01} as well 
as the construction of a \emph{bad field} \cite{BHMW07}. The green field is obtained using Hrushovski's 
amalgamation method (without collapse), whereas the bad field is constructed by collapsing the former to a field of finite Morley rank. (We refer to Section \ref{S:Other} for a more systematic treatment of this amalgamation method.)

In both constructions, uniformity results for intersections of tori 
with algebraic varieties in characteristic 0 have to be used in order to establish the necessary 
definability properties which make Hrushovski's amalgamation method work. We recall these uniformity 
results (called \emph{weak CIT}) since they will be used in our construction of a geometric 
genericity notion in the green field and also in the proof that the bad fields have the DMP (see Section \ref{S:Gen_Green}).

\subsection{Dimension, codimension and predimension}
In the following we gather the results which will be needed to get a definable control on the (pre-)dimension 
in the green fields.

Let us fix some notation (mainly following \cite{BHMW07}):  $\mathbb{C}$ denotes an algebraically closed field of characteristic $0$. A 
variety $V$ will always be a closed subvariety of some $\Gm^{n}$ (which may be identified with the set $\Cn{n}$ of 
its $\mathbb{C}$-rational points). A \emph{torus} is  a connected algebraic subgroup of $\Gm^n$. It is 
described by finitely many equations of the form: $x_1^{r_1}\cdot\ldots\cdot x_n^{r_n}=1$, where 
$r_i\in\Z$. If $T$ is a torus and $\a$ is generic in $T$ over $\mathbb{C}$, then the $\Q$-linear dimension 
of $\a$ over $\mathbb{C}^*$ (modulo torsion) equals the algebraic dimension of $T$ (as a variety) and will be denoted 
by $\ldim (T)$ or $\dim (T)$. Given a closed and irreducible subvariety $W$ in $\Gm^n$, its 
\emph{minimal torus} is the smallest torus $T$, such that $W$ lies in some coset $\a\cdot T$.
 In this case, we define its \emph{codimension} $\cd(W):=\dim(T)-\dim(W)=\ldim(W)-\dim(W)$, 
where $\ldim(W):=\dim(T)$. The \emph{predimension} of $W$ is given by $\delta(W):=2\dim(W)-\dim(T)=\dim(W)-\cd(W)$.
 
An irreducible subvariety  $W\subseteq V$ is \emph{$\cd$-maximal in $V$} if $\cd(W')>\cd(W)$ for every irreducible subvariety  $W\subsetneq W'\subseteq V$. Clearly, irreducible components of $V$ and cosets of tori maximally contained in $V$ are examples of $\cd$-maximal subvarieties.

We now present a result which was stated by Poizat \cite[Corollaire 3.7]{Po01}. It is a reformulation of a result proved by Zilber \cite{Zi02} (and later generalised by 
Kirby \cite{Ki09} to the context of semiabelian varieties).

\begin{Fact}\label{T:Poizat}  Let $\mathcal{V}=\{V_{\b}\,|\,\b\models\theta(\z)\}$ be a uniformly definable family of closed subvarieties of $\Gm^{n}$. There exists a finite collection of tori $\mathcal{T}(\mathcal{V})=\{T_0,\dots,T_r\}$, such that for any member $V_{\b}$ of the 
 family and any $\cd$-maximal subvariety $W$ of $V_{\b}$, the minimal torus of $W$ belongs 
to $\mathcal{T}(\mathcal{V})$.
\end{Fact} 

This property, which Zilber called \emph{weak CIT}, is at the origin of a series of definability results, as we will see in the sequel.

A matrix $M=(m_{i,j})\in \mathrm{Mat}(n\times n,\Z)$ acts on $\Gm^{n}$. For $\a\in \Gm^{n}$, we put 
$$\a^M:=\left(\prod_{j=1}^n a_j^{m_{1,j}},\ldots,\prod_{j=1}^n a_j^{m_{n,j}}\right).$$ 
\begin{Defi}
Let $V\subseteq\Gm^{n}$ be an irreducible variety defined over the algebraically closed field $K$. 
The variety $V$ is called \emph{free} if its minimal torus is 
 equal to $\Gm^{n}$. It is called \emph{rotund}\footnote{In Zilber's terminology \cite{Zi04}, our notion of rotund corresponds to `$G$-normal' and `$G$-free'; the 
 term `rotund' is taken from \cite{Ki09}. Since we only use rotund varieties which 
 are free as well, we include the freeness condition in our definition.} if it is 
 free and if for any  $K$-generic tuple $\a$ in $V$ and any 
 $M\in \mathrm{Mat}(n\times n,\Z)$, putting $W:=\Loc(\a^M/K)$, one has $\delta(W)\geq0$.
\end{Defi}

A property $\mathcal{P}$ of algebraic varieties is called \emph{definable} if for any uniformly definable family of algebraic varieties 
$\mathcal{V}=\{V_{\b}\,|\,\b\models\theta(\z)\}$, the set of parameters $\b$ such that $V_{\b}$ has the 
property $\mathcal{P}$ is definable.

\begin{Fact}[\cite{Ki09}] \label{F:RotundDef}
\begin{enumerate}
\item[(1)] Freeness is a definable property.
\item[(2)] Rotundity is a definable property.
\end{enumerate}
\end{Fact}
\begin{proof}
For convenience, we include the argument.
Let $V_{\b}$ be an instance of a uniformly definable family $\mathcal{V}$ of irreducible varieties 
in $\Gm^n$. Since the minimal torus of $V_{\b}$ lies in the finite collection of tori $\{T_0,\dots,T_r\}$ attached to $\mathcal{V}$ it is sufficient to avoid all $T_i\neq\Gm^{n}$ from this 
collection to force the minimal torus of $V_{\b}$ to be equal to $\Gm^{n}$. This can be done definably and shows (1).

To prove (2), we may assume that the dimension of $V_{\b}$ is equal to $k$ throughout the family, and 
that all instances are free. Thus, $2k-n=d\geq0$, and it suffices to impose the following crucial condition:
\begin{itemize}
\item[$(*)$] \emph{For generic $\overline{g}\in V_{\b}$ and $T\in\mathcal{T}(\mathcal{V})$ and any irreducible 
component $W$ of $V_{\b}\cap\overline{g}\cdot T$ of maximal dimension, $\dim(W)-\cd(W)\leq d$ holds.}
\end{itemize}
The finiteness of $\mathcal{T}(\mathcal{V})$ implies that $\cd$ is definable. It is well known that $\dim$ is definable as 
well. Using definability of types in $ACF$, it follows that $(*)$ is a definable condition. It is not hard to see that $(*)$ is enough to guarantee rotundity of $V_{\b}$ (cf.\ the proof of \cite[Lemma 4.3]{BHMW07}).
\end{proof}

Let us mention that freeness is also a definable property in positive characteristic. One may prove this using Zilber's Indecomposibility Theorem. (We thank Martin Bays for pointing this out to us.)

\begin{Lem}\label{L:MaxDelta}
Let $V$ be an irreducible subvariety of $\Gm^n$ and let $\mathcal{T}(V)$ 
be the finite family of tori given in Fact \ref{T:Poizat}. Assume that $V$ is free. Let $W\subsetneq V$ be 
a proper irreducible subvariety such that $\delta(W)\geq\delta(V)$. 

Then the minimal torus of $W$ is contained 
in some $T\in\mathcal{T}(V)$ with $T\subsetneq\Gm^n$.
\end{Lem}

\begin{proof}
From $\dim(W)<\dim(V)$ and $\delta(W)\geq\delta(V)$ we infer $\cd(W)<\cd(V)$. Let $W'$ be $\cd$-maximal such that $W\subseteq W'\subseteq V$ and $\cd(W')\leq\cd(W)$. The minimal torus of $W$ 
is contained in the minimal torus $T$ of $W'$. Clearly $T$ is a proper subtorus of $\Gm^n$; 
moreover, $T\in\mathcal{T}(V)$ by Fact \ref{T:Poizat}.
\end{proof}

\subsection{Green colour, green fields of Poizat and bad fields}
We now recall the construction of the green field of Poizat \cite{Po01} and of the bad field \cite{BHMW07}. 
We expand the language of rings by a new unary predicate $\UE$ and thus work in 
$\mathcal{L}=\mathcal{L}_{rings}\cup\{\UE\}$. Elements in $\UE$ will be called \emph{green}, those not 
in $\UE$ are \emph{white}. We consider $\L$-structures of the form $K=(K,+,-,\times,0,1,\UE(K))$ 
such that $K$ is an algebraically closed field of characteristic 0 and $\UE(K)$ a subgroup of the multiplicative group $K^{\times}$ which is divisible and torsion free. So $\UE$ is a vector space over 
$\Q$. If we write $K\subseteq L$, we mean that $K$ is an $\L$-substructure of $L$, in particular 
$\UE(L)\cap K=\UE(K)$. 

\smallskip

We call $K$ a \emph{green field} if $\delta(k)=2\tr(k)-\ldim(\UE(k))\geq0$ for every $k\subseteq K$ of finite transcendence degree. Here $\delta(k)$ is called the \emph{predimension} 
of 
$(k,\UE(k))$. More generally, for $K\subseteq L$ such that $\tr(L/K)$ and $\ldim(\UE(L)/\UE(K))$ are 
finite, we put 
$$\delta(L/K)=2\tr(L/K)-\ldim(\UE(L)/\UE(K)).$$
An extension $K\subseteq L$ of green fields is called \emph{self-sufficient} if 
$\delta(L'/K)\geq 0$ for every green field $L'$ such that $K\subseteq L'\subseteq L$ and $\tr(L'/K)$ is finite; we write $K\leq L$ if this holds. If $A\subseteq L$ is any subset, there is a minimal (with respect to inclusion) green field $K'$ such that $A\subseteq K'\leq L$; it is called the \emph{self-sufficient closure of $A$ in $L$} and denoted by $\cl_{\omega}(A)=\cl_{\omega}^L(A)$. 
Note that this notion depends on $L$, but often we will omit the superscript if $L$ is clear from the 
context. If $A$ contains a $\Q$-basis of the green points of its self-sufficient closure in $L$, we also write 
$A\leq L$ (by a slight abuse of notation), and $A$ is called a \emph{self-sufficient subset} 
of $L$. Note that in this case $\cl_{\omega}^L(A)$ is given by $A^{alg}$, the algebraic closure of 
$A$ in the field sense.

\smallskip
If $\a$ is a finite tuple from $L$ and $B\subseteq L$, the \emph{dimension of $\a$ over $B$} is given by $\d(\a/B)=\d^L(\a/B)=\delta\left(\cl_{\omega}(B\a)/\cl_{\omega}(B)\right)$. Note that if 
$K\subseteq A\subseteq L$ for some $K\leq L$ with $\tr(A/K)<\infty$, then $\tr\left(\cl_{\omega}(A)/K\right)<\infty$ as well.

\smallskip
The following lemma is a direct consequence of the definitions. 

\begin{Lem}\label{L:RotundSelfsuff}
Let $K\subseteq L$ be an extension of green fields. Assume 
that $(g_1,\ldots,g_n)$ is a basis of $\UE(L)$ over $\UE(K)$. Then $K\leq L$ if and only 
if $\Loc(\overline{g}/K)$ is rotund.\qed
\end{Lem}

Let $(\C,\leq)$ be the class of green fields with self-sufficient embeddings.

\begin{Fact}[\cite{Po01}]\label{F:PoizatGreen}
\begin{enumerate}
\item[(a)] The class $\C$ is elementary, and $(\C,\leq)$ has the amalgamation property (AP) and 
the joint embedding property (JEP). Moreover, up to $\L$-isomorphism, the subclass $\C^{fin}$ of green 
fields of finite transcendence degree is countable. 
\item [(b)] Let $K_{\omega}$ be the Fra\"{i}ss\'e-Hrushovski limit of $(\C^{fin},\leq)$. Then $K_{\omega}$ 
is a saturated model of its $\L$-theory $T_{\omega}$. \item[(c)] The algebraic closure in $T_{\omega}$ equals the self-sufficient closure.  
\item[(d)] Let $A,A'\subseteq K\models T_{\omega}$. Then 
$\tp_{T_{\omega}}(A)=\tp_{T_{\omega}}(A')\Leftrightarrow\cl_{\omega}(A)\simeq_{\L}\cl_{\omega}(A')$ (over the map sending $A$ 
to $A'$).
\item[(e)] The theory $T_{\omega}$ is $\omega$-stable of Morley rank $\omega\cdot2$, with $\RM(\UE)=\omega$. Moreover, $\RM(\a/A)<\omega\Leftrightarrow\d(\a/A)=0$ for all sets $A$ and finite tuples $\a$.
\item[(f)] Let $K\models T_{\omega}$. Then any non-zero element of $K$ may be written in the form $(a+b)\times(c+d)$ for some green elements $a,b,c,d$.
In particular $K$ is in the definable closure of $\UE(K)$.
\end{enumerate} 

\end{Fact}

\begin{Rem}
Assuming Schanuel's Conjecture, Zilber shows in \cite{Zi04} that there is a natural model of $T_{\omega}$, 
namely the structure $(\mathbb{C},+,\times,\UE)$, where the set of green points is given by $\UE:=\{exp\left( t(1+i) + q \right) \,\mid\, t \in \R, \, q \in \Q\}$.
\end{Rem}


The construction of Poizat provides a ``bad field of infinite rank" in characteristic 0. It is possible to collapse the theory $T_{\omega}$ to obtain a \emph{bad field}, i.e.\ a field of finite Morley rank with 
a definable infinite proper subgroup of the multiplicative group of the field. In \cite{BHMW07}, Baudisch, Martin-Pizarro, Wagner and the author construct an elementary 
subclass $\Cmu\subseteq\C$ such that $(\Cmu,\leq)$ has (AP) and (JEP). Let $K_{\mu}$ be the 
corresponding Fra\"{i}ss\'e-Hrushovski limit, and $T_{\mu}$ its $\L$-theory.

\begin{Fact}[\cite{BHMW07}]\label{F:BadField}
\begin{enumerate}
\item[(a)] $K_{\mu}$ is saturated.
\item[(b)] Let $A,A'\subseteq K\models T_{\mu}$. Then 
$\tp_{T_{\mu}}(A)=\tp_{T_{\mu}}(A')\Leftrightarrow\cl_{\omega}(A)\simeq_{\L}\cl_{\omega}(A')$ (over the map sending $A$ to $A'$).
\item[(c)] The theory $T_{\mu}$ is of Morley rank 2, and $\UE$ is strongly minimal.
\item[(d)] For all $A\subseteq K\models T_{\mu}$ and $\a\in K$ one has $\d(\a/A)=\RM(\a/A)$. 
In particular, $\acl_{\mu}(A)=\{a\in K\,\mid\,\d(a/A)=0\}$, where $\acl_{\mu}$ denotes the 
algebraic closure in $T_{\mu}$.  
\item[(e)] Let $K\models T_{\mu}$. Then any non-zero element of $K$ may be written in the form $(a+b)\times(c+d)$ for some green elements 
$a,b,c,d$. In particular $K$ is in the definable closure of $\UE(K)$, so $T_{\mu}$ is almost strongly minimal.
\item[(f)] $T_{\mu}$ is model-complete. 

\end{enumerate}
\end{Fact}

The following result will thus apply to the theory $T_{\mu}$.

 \begin{Fact}[\cite{Wa01}] \label{F:Bad_IE}Let $K$ be a field of finite Morley rank (in some expansion $\L$ of the language of rings).
 \begin{enumerate}
\item[(a)] Any algebraically closed subset of $K$ is an elementary substructure. 
\item[(b)] The theory $\Th_{\L}(K)$ eliminates imaginaries.
\end{enumerate}
 \end{Fact}

We mention another fact which will be needed later on. It is a direct consequence of \cite[Lemma 10.3(2)]{BHMW07}.

\begin{Fact}\label{F:DimTypeDef}
Work in $T_{\omega}$ or in $T_{\mu}$. For any $d\geq0$ and any variable tuples $\x$ and $\z$ there is a partial type $\pi_d(\x,\z)$ such that for any $\a$ and $\b$ one has $\models\pi_d(\a,\b)$ if and only if $\d(\a/\b)\geq d$.
\end{Fact}

 We finish this section with an example showing that in both $T_{\omega}$ and $T_{\mu}$, we cannot 
 infer from the characterisations of types in Fact \ref{F:PoizatGreen}(d) and Fact \ref{F:BadField}(b), respectively, that two self-sufficient green tuples having the same field type (over an algebraically closed and self-sufficient base) must have the same type. The problem is that one has to \emph{choose 
 green roots}.
 
 \begin{Ex}\label{Ex:ChoosingRoots}
 Let $L$ be a model of $T_{\omega}$ or $T_{\mu}$, $K=\Q^{alg}\subseteq L$, and let $\a=(a_1,a_2,a_3),\a'=(a'_1,a'_2,a'_3)\in L$ be green tuples. Put $A=K(\a)^{alg}$, $A'=K(\a')^{alg}$. 
 Suppose that $A$ and $A'$ 
 are self-sufficient in $L$, and that $\a$ is a $\Q$-basis of $\UE(A)$ over $K$, similarly for $\a'$ and $\UE(A')$ . Suppose that both 
 $\a$ and $\a'$ are generic in the variety $V$ given by the equation $X=(Y+Z)^2$. Note that exactly 
 one of the two square roots of $a_1$ (and of $a_1'$) is green. Suppose that $a_2+a_3$ and 
 $-a_2'-a'_3$ are green. Then $\a$ and $\a'$ do not have the same type over $K$ (not even over $\emptyset$).
 \end{Ex}

\section{A definability result for algebraic varieties}\label{S:def}\newcounter{Var_enum}
In this section we prove a definability result for varieties in characteristic 0 which will allow us 
to deal with uniformity issues around multiplicity in green fields: it is the major ingredient to show that 
the bad fields constructed in \cite{BHMW07} have the DMP and that the green fields of Poizat admit 
a geometric notion of genericity.

\begin{Defi}
\begin{enumerate} 
\item Let $L/K$ be a field extension with $K\models \mathrm{ACF}_0$, and let $l\geq2$ be an integer. A tuple $\a=(a_1,\ldots,a_n)$ from 
$L^{\times}$ is called \emph{$l$-Kummer generic over $K$} if $\Gal\left(K(\sqrt[l]{a_1},\ldots,\sqrt[l]{a_n})/K(\a)\right)\simeq(\Z/l\Z)^n$. 

The tuple $\a$ is called \emph{Kummer generic over $K$} if it is $l$-Kummer generic over $K$ for every $l\geq2$.
\smallskip

\item Let $V\subseteq \Gm^{n}$ be an irreducible closed subvariety of the standard 
torus $\Gm^{n}$, $V$ defined over $K\models\mathrm{ACF}_0$. The variety $V$ is called \emph{$l$-Kummer generic} (\emph{Kummer generic}, resp.) if every tuple $\a=(a_1,\ldots,a_n)$ which is generic in $V$ over 
$K$ is $l$-Kummer generic (Kummer generic, resp.) over $K$.
\end{enumerate}
\end{Defi}

The notion of a Kummer generic tuple is taken from \cite{Zi06}, although Zilber calls such a tuple 
\emph{simple}. Note that the definition of a Kummer generic variety does not depend on the choice of the algebraically closed 
field $K$.

\smallskip

Let $A$ be an abelian group, $B$ be a subgroup of $A$ and $l\geq 2$ a natural number. Recall that $B$ is an \emph{$l$-pure} 
subgroup of $A$ if whenever the equation $lx=b$ has a solution in $A$, where $b\in B$, then it has already a solution in $B$. If $B$ is $l$-pure in $A$ for every $l$, it is called a pure 
subgroup. Note that if $\mathrm{Tor}(A)\subseteq B$, then $B$ is $l$-pure in $A$ if and only if $A/B$ has trivial $l$-torsion.

\smallskip

For a field extension $L/K$ and $X\subseteq L^{\times}$ we denote $K^{\times}\langle X\rangle$ the 
subgroup of $L^{\times}$ generated by $K^{\times}\cup X$. Let $M$ be an algebraically closed field and 
$\Gamma$ a subgroup of the multiplicative group $M^\times$ of $M$. Then the divisible hull of $\Gamma$ (i.e.\ the set of elements $m\in M^{\times}$ such that 
$m^n\in\Gamma$ for some $n\geq1$) is denoted by $\divh(\Gamma)$.

\begin{Fact}\label{F:GaloisKummer}
Let $K$ be an algebraically closed field of characteristic 0 and $V\subseteq\Gm^n$ be a closed irreducible subvariety defined over $K$.
\begin{enumerate}
\item Let $L/K$ be a field extension, $\a$ an $n$-tuple from $L^{\times}$ and $l\geq 2$ an integer. The following are equivalent:

\begin{enumerate}
\item $\a$ is $l$-Kummer generic over $K$.
\item $\a$ is multiplicatively independent over $K^{\times}$ and $\left(L_0^{\times}\right)^l\cap \langle a_1,\ldots,a_n\rangle=\langle a_1^l,\ldots,a_n^l\rangle$, where 
$L_0=K(\a)$ and $\left(L_0^{\times}\right)^l=\{b^l\,\mid\, b\in L_0^{\times}\}$. 
\item  The elements $a_1/K^{\times},\ldots,a_n/K^{\times}$ generate an $l$-pure subgroup of rank $n$ inside 
the group $K(\a)^{\times}/K^{\times}$.
\item If $\alpha_i$ is an $l$-th root of $a_i$, then $\tp_{\mathrm{ACF}_0}(\alpha_1,\ldots,\alpha_n/K\a)$ is determined 
by $\{x_i^l=a_i\}_{1\leq i\leq n}$.
\end{enumerate}

\medskip

\item The following are equivalent:

\begin{enumerate}
\item $V$ is $l$-Kummer generic.
\item The variety $\sqrt[l]{V}\subseteq\Gm^n$ given by ``$(X_1^l,\ldots,X_n^l)\in V$" is irreducible.
\end{enumerate}

\medskip

\item Let $\a$ be generic in $V$ over $K$. Then, the following are equivalent:

\begin{enumerate}
\item $V$ is Kummer generic.
\item Any group automorphism of $\divh(K^{\times}\langle\a\rangle)$ fixing $K^{\times}\langle\a\rangle$ pointwise 
lifts to a field automorphism of $K(\a)^{alg}$, i.e.\ the natural map (given by restriction) $\Gal\left(K(\a)\right)\rightarrow \Aut_{gp}\left(\divh(K^{\times}\langle\a\rangle)/K^{\times}\langle\a\rangle\right)$
is surjective.
\item $V$ is $p$-Kummer generic for every prime number $p$.
\item The elements $a_1/K^{\times},\ldots,a_n/K^{\times}$ generate a pure subgroup of rank $n$ of $K(\a)^{\times}/K^{\times}$.
\end{enumerate}
\end{enumerate}
\end{Fact}

\begin{proof}
Note that, with the notation from (1.b), letting $A:=(L_0^\times)^l\langle a_1,\ldots,a_n\rangle$, one has $A/(L_0^\times)^l\cong\langle a_1,\ldots,a_n\rangle/\langle a_1,\ldots,a_n\rangle\cap (L_0^\times)^l$, 
so (a)$\iff$(b) in (1) follows from Kummer theory (see e.g.\ \cite[VI. Thm 8.1]{La93}). The other equivalences are easily verified. 

In (2), note that if $W$ is an irreducible component of maximal dimension of $\sqrt[l]{V}$, then all the other irreducible components are multiplicative 
translates of $W$ by some $l$-torsion element $\zeta\in\Gm^n$. In particular, $\sqrt[l]{V}$ is equidimensional. Now, (2) follows, using (a)$\Leftrightarrow$(d) in (1). Part (3) is left to the reader. 
\end{proof}


The pathology we encountered in Example \ref{Ex:ChoosingRoots} does not exist in case the tuples are Kummer generic, as is shown by the following corollary.

\begin{Cor}\label{C:KummerSameType}
Let $K=K^{alg}\leq L\models T$, where $T$ equals $T_{\omega}$ or $T_{\mu}$. Let $\a,\a'\in L$ be such that $K\a\leq L$ and $K\a'\leq L$. Suppose that $\a$ and $\a'$ are coloured in the same way and 
that $\a$ and $\a'$ have the same field type over $K$. Moreover, suppose that $\a$ is Kummer generic over $K$. Then $\tp_T(\a/K)=\tp_T(\a'/K)$.
\end{Cor}

\begin{proof}
Note that since $K\a\leq L$, we have $\cl_{\omega}(K\a)=K(\a)^{alg}$ (similarly for $\a'$).
Choose a field isomorphism $\alpha:K(\a)^{alg}\simeq K(\a')^{alg}$ extending the map $K\a\mapsto K\a'$. 
We have $\alpha(\divh(K^{\times}\langle\a\rangle))=\divh(K^{\times}\langle\a'\rangle)$, and it is easy to see that there 
exists $\sigma_0\in\Aut_{gp}\left(\divh(K^{\times}\langle\a\rangle)/K^{\times}\langle\a\rangle\right)$ such 
that 
$$\alpha_0\circ\sigma_0:\left(\divh(K^{\times}\langle\a\rangle),\UE\cap\divh(K^{\times}\langle\a\rangle)\right)\simeq
\left(\divh(K^{\times}\langle\a'\rangle),\UE\cap\divh(K^{\times}\langle\a'\rangle)\right)$$
is an isomorphism of groups respecting the green points (here $\alpha_0$ denotes the map $\alpha\upharpoonright_{\divh(K^{\times}\langle\a\rangle)}$). Since $\a$ contains a basis of $\UE\cap K(\a)^{alg}$ over $\UE(K)$, necessarily $\UE\cap K(\a)^{alg}=\UE\cap\divh(K^{\times}\langle\a\rangle)$ (similarly for $\a'$). 
By Fact \ref{F:GaloisKummer} there exists $\sigma\in\Gal\left(K(\a)\right)$ restricting to $\sigma_0$, 
and we obtain an $\L$-isomorphism
$$\alpha\circ\sigma:\left(K(\a)^{alg},\UE\cap K(\a)^{alg}\right)\simeq\left(K(\a')^{alg},\UE\cap K(\a')^{alg}\right).$$

The result follows, using Fact \ref{F:PoizatGreen}(d) or Fact \ref{F:BadField}(b), respectively.
\end{proof}

\begin{Fact}[{\cite[Lemma 2.1]{Zi06}}]     \label{F:FreeAb}
Let $K$ be an algebraically closed field and $L/K$ a finitely generated field extension. Then $L^{\times}/K^{\times}$ 
is a free abelian group. 
\end{Fact}

This fact is proved by embedding $L^{\times}/K^{\times}$ into the group of Weil divisors 
of a suitably chosen variety $V$ such that $K(V)=L$. In the proof of the following key 
definability result, we will give an effective version of this argument.

\begin{Prop}\label{P:SimpleDef}
Being a Kummer generic (irreducible) variety is a definable 
property in characteristic 0.
\end{Prop}

\begin{proof}
Let $\mathcal{V}=\{V_{\b}\,|\,\models\theta(\b)\}$ be a uniformly definable family of closed subvarieties of $\Gm^{n}$. If $V_{\b}$ is not Kummer generic, it is easy to construct a formula $\theta'(\z)\in\tp(\b)$ such that 
whenever $\models\theta'(\b')$, then $V_{\b'}$ is not Kummer generic. By compactness, it thus suffices 
to construct, for every tuple $\b_0$ such that $V_{\b_0}$ is Kummer generic, some $\theta_0(\z)\in\tp(\b_0)$ such that $V_{\b'}$ is Kummer generic whenever $\models\theta_0(\b')$.

So assume $V_{\b_0}$ to be Kummer generic. Suppose that $V_{\b_0}$ is irreducible of dimension $d$. Choose $I\subseteq\{1,\ldots,n\}$, $|I|=d$, such 
that for generic $\a$ in $V_{\b_0}$ (over $K\models \mathrm{ACF}_0$ containing $\b_0$) one has 
$\a\in K(\a_I)^{alg}$, with $\a_I=(a_i)_{i\in I}$ (i.e.\ $\a_I$ is a transcendence basis of 
$K(V_{\b_0})=K(\a)$ over $K$). 
Strengthening $\theta$ and choosing appropriate natural numbers $m$ and $k$, we may assume that every variety $V=V_{\b}$ from the family 
$\mathcal{V}$ satisfies the following conditions (below, we will always work over an algebraically closed 
field $K$ over which $V$ is defined):
\begin{enumerate}[\upshape(a)]\setcounter{enumi}{\value{Var_enum}} 

\item \label{irred} $V$ is irreducible of dimension $d$.

\item \label{alg} Let $\a$ be generic in $V$. Then $\a_I$ is a transcendence basis of $K(V)=K(\a)$ over 
$K$. Moreover, $[K(\a):K(\a_I)]\leq m$.

\item \label{deg} Let $\a$ be generic in $V$, and let $\varepsilon:I\rightarrow\{-1,1\}$ be any function. Denote the tuple $(a_i^{\varepsilon(i)})_{i\in I}$ by $\a_I^{\varepsilon}$. 
Then any $a_j$ satisfies a polynomial equation over 
$K(\a_I^{\varepsilon})$ of the form

$$Y^m+\frac{f_{m-1}(\a_I^{\varepsilon})}{g_{m-1}(\a_I^{\varepsilon})}Y^{m-1}+\ldots+  \frac{f_0(\a_I^{\varepsilon})}{g_0(\a_I^{\varepsilon})}=0,$$
where $f_l(\a_I^{\varepsilon})$ and $g_l(\a_I^{\varepsilon})$ are polynomials in $\a_I^{\varepsilon}$ of total degree at most $k$ (for 
$0\leq l<m$).

\item \label{Ax} Let $\a$ be generic in $V$. Then 
$\a$ is multiplicatively independent over $K^{\times}$.
\end{enumerate}

By standard arguments we may achieve (\ref{irred}), (\ref{alg}) and (\ref{deg}). The property 
(\ref{Ax}) is definable by Fact \ref{F:RotundDef}(1).

\smallskip
\noindent
Claim. \emph{For a given $l\geq2$, there is $\theta_l(\z)\in\tp(\b_0/K)$ such that for any 
$\b$ with $\models \theta_l(\b)$ the variety $V_{\b}$ is $l$-Kummer generic.}

By Fact \ref{F:GaloisKummer}, $V_{\b}$ is $l$-Kummer generic if and only if the variety defined by the condition 
 ``$(X_1^l,\ldots,X_n^l)\in V_{\b}$" is irreducible. This proves the claim, for the latter condition is definable in $\b$.
 
 \smallskip
 
 Since a variety is Kummer generic if it is $p$-Kummer generic for every prime number $p$, using the previous claim, the 
 proof of the proposition is thus finished once the following lemma is established. 
 \end{proof}
 
 \begin{Lem}
 Let $V\subseteq\Gm^n$ be as above, satisfying (\ref{irred}-\ref{Ax}) (with $m$ and 
 $k$ as in (\ref{alg}) and (\ref{deg}), respectively). Let $p>n!m^nk^n$ be a prime number. Then $V$ 
 is $p$-Kummer generic.
 \end{Lem}
 \begin{proof}
 
 (i) We consider the \emph{group of (Weil) divisors}\footnote{Alternatively, the classical and 
 more geometric way would be to work with the group of Weil divisors of a certain projective variety $V'$, namely the normalisation of $\mathbb{P}^d$ in the field $K(V)\supseteq K(\mathbb{P}^d)=K(\a_I)$.} of the 
 function field $K(V)/K$, given by
 
 $$\Div(K(V)/K):=\bigoplus_{v\in\RegP}\Z\cdot v,$$
 where $\RegP=\RegP(K(V)/K)$ denotes the set of all discrete valuations of $K(V)$ which are trivial on $K$ and such that the residue 
 field is of transcendence degree $d-1$ over $K$. 
 
 For any $f\in K(V)$ there is only a finite number of $v\in \RegP(K(V)/K)$ such that $v(f)\neq0$, and one has $f\in K$ if 
 and only if $v(f)=0$ for all $v\in \RegP(K(V)/K)$. This follows from standard arguments in valuation theory. (We refer to sections VI.\S 14 and VII.\S 4bis in \cite{ZS60}.) 
 The following map is thus a group homomorphism which induces an embedding 
 of $K(V)^{\times}/K^{\times}$ into $\Div(K(V)/K)$:
 $$K(V)^{\times}\rightarrow\Div(K(V)/K),\quad f\mapsto(f):=\sum_{v}v(f)\cdot v$$
 
 \smallskip
 
 (ii) Let $v'$ be in $\RegP(K(\a_I)/K)$. Suppose $v'(a_i)\geq0$ 
 for all $i\in I$. (Replacing $\a_I$ by a suitable $\a_I^{\varepsilon}$, we may always achieve this.) It follows from the assumption on $v'$ that the ideal of $K[\a_I]$ given by the elements of positive valuation is a prime ideal of height 1, so equal to $(P)$ for some irreducible polynomial $P=P(\a_I)$. This means 
 that $v'\left(P^z\frac{f(\a_I)}{g(\a_I)}\right)=z$ for all $f$ and $g$ which are not divisible by $P$. 
 Denote this valuation by $v'_P$.
 
 \smallskip
  
 (iii) Let $v\in\RegP(K(\a)/K)$, and let $v'$ be its restriction to $K(\a_I)/K$ (which 
 is an element of $\RegP(K(\a_I)/K)$ by standard properties of algebraic extensions 
 of valuations.) By (ii), we may assume that $v'=v(P)v'_P$ for some irreducible 
 polynomial $P=P(\a_I)$. 
 
 We will show that $|v(a_j)|\leq mk$ for all $j\leq n$. By (\ref{deg}),
 $$a_j^m+\frac{f_{m-1}(\a_I)}{g_{m-1}(\a_I)}a_j^{m-1}+\ldots+  \frac{f_0(\a_I)}{g_0(\a_I)}=0,$$
 so there exists $r<m$ such that $v(a_j^m)=v\left(\frac{f_{r}(\a_I)}{g_r(\a_I)}a_j^r\right)$ and thus 
 $$(m-r)v(a_j)=v\left(\frac{f_{r}(\a_I)}{g_r(\a_I)}\right).$$ By (\ref{alg}) and the fundamental inequality, $|v(f)|\leq m|v'_P(f)|$ for any $f\in K(\a_I)$.
 
 Moreover, since the total degrees of $f_{r}$ and $g_{r}$ are bounded by $k$ (by (\ref{deg})), it follows that $|v'_P\left(\frac{f_{r}(\a_I)}{g_r(\a_I)}\right)|\leq k$. 
 Thus, $|v(a_j)|$ is bounded by $mk$.
 
 \smallskip
 
 (iv) Consider the elements $(a_1),\ldots,(a_n)$ of $\Div(K(V)/K)$. By (\ref{Ax}) and (i), they are 
 linearly independent over $\Z$, so 
 there are valuations $v_1,\ldots,v_n\in\RegP(K(V)/K)$ such that the square matrix 
 $M=\left(v_i(a_j)\right)_{i,j}$ has non-zero determinant. Now, $|\det(M)|\leq n!m^nk^n$ by (iii) and the Leibniz formula, so $\det(M)\not\equiv 0\mod p$ (as $p>n!m^nk^n$ by assumption). 
 This means that no 
 element of the form $\sum_{i=1}^nr_i(a_i)$, with $0\leq r_i<p$ not all 0, is divisible by $p$ in 
 $\Div(K(V)/K)$. It follows that $\prod_{i=1}^n a_i^{r_i}$ does not have a $p$-th root in $K(\a)$. By Fact \ref{F:GaloisKummer}, this shows 
 that $V$ is $p$-Kummer generic.
 \end{proof}

\begin{Rem}\label{R:Gabber}
Gabber suggested a completely different proof for definability of Kummer genericity, a proof which generalises to semiabelian varieties in arbitrary 
characteristic. 

In joint work with Bays and Gavrilovich \cite{BGH11}, we extract the `Galois theoretic' essence of Gabber's argument and give 
a model-theoretic proof which applies to any definable abelian group of finite Morley rank with the definable multiplicity property.
\end{Rem}

Before we finish this section, let us mention an important corollary of Proposition \ref{P:SimpleDef}. It was observed by Roche that there is a gap in the construction of the bad field as it is given in \cite{BHMW07}. 
The reason for 
this is intimately related to the problem raised in Example \ref{Ex:ChoosingRoots}. In fact, \cite[Bemerkung 6.7]{BHMW07} is not true in general, and so the 
proof of the economic amalgamation lemma \cite[Satz 9.2]{BHMW07} is not correct. Fortunately, we may provide the necessary technical improvement --- the existence of 
strongly minimal codes --- in Corollary \ref{C:SmCodes} below, so that the proof of the economic amalgamation lemma goes through without any changes. 

\smallskip

In his thesis \cite{Ro11}, Roche considers so-called \emph{octarine fields}, certain expansions of abelian varieties by a predicate for a non-algebraic subgroup, a context which is similar to bad fields. 
It is explained in detail there how strongly 
minimal codes are used to prove the economic amalgamation lemma. The same arguments apply in the context of bad fields.

\begin{Cor}  \label{C:SmCodes}
There is a collection of codes satisfying all the requirements of \cite[Definition 4.7]{BHMW07} and moreover that 
the instances of any code are strongly minimal definable sets.
\end{Cor}

\begin{proof}
We adopt the terminology and notation from \cite[Definition 4.7]{BHMW07}, restricting our attention to minimal 
prealgebraic formulas $\phi(\x)$ such that the corresponding variety is Kummer generic (equivalently, 
any generic solution of $\phi$ over an algebraically closed field is Kummer generic). 
We strengthen the definition of a code 
$\phi_{\alpha}(\x,\z)$ by adding that for any non-empty instance $\phi_{\alpha}(\x,\b)$, its 
Zariski closure $V_{\alpha}(\x,\b)$ is a Kummer generic variety (this is a definable property by Proposition \ref{P:SimpleDef}). 

It follows from Corollary \ref{C:KummerSameType} that $\phi_{\alpha}(\x,\z)\wedge\bigwedge_i \UE(x_i)$ is a 
family of strongly minimal sets.
\end{proof}

\section{Generic automorphisms of green and bad fields}\label{S:Gen_Green}\newcounter{Gen_enum}
In this section, we will establish the axiomatisability of the generic automorphism in the 
green and bad fields. We use the notation from Section \ref{S:delta}.

\subsection{Generic automorphisms of the green field of Poizat}

\begin{Lem}\label{L:GreenGenNotion}
The theory $T_{\omega}$ admits a geometric notion of genericity.
\end{Lem}

\begin{proof}
Consider a type $\tp(\tilde{a}/K)$ where $\tilde{a}$ is a finite tuple from 
$\mathfrak{C}\elres K\models T_{\omega}$ of the form 
$\overline{g}\overline{g'}\overline{w}\overline{w'}$ (maybe after reordering), satisfying 
the following conditions:
\begin{itemize}
\item[(i)] The elements from $\overline{g}\overline{g'}$ are green, those from $\overline{w}\overline{w'}$ are white.
\item[(ii)] $K\tilde{a}\leq\mathfrak{C}$, and $\overline{g}=(g_1,\ldots,g_n)$ is a basis of 
$\UE(K(\tilde{a})^{alg})$ over $\UE(K)$ such that $\overline{g}'\in\langle\UE(K)\overline{g}\rangle$.
\item[(iii)] $\overline{w}$ is multiplicatively independent over $K\overline{g}$, and 
$\overline{w}'\in K^{\times}\langle \overline{g}\overline{w}\rangle$.
\item[(iv)] $\overline{w}\in K[g_1,\ldots,g_n,\frac{1}{g_1},\ldots,\frac{1}{g_n}]$.
\item[(v)] $\overline{g}$ is Kummer generic over $K$.
\end{itemize}

Call a type \emph{special} if it satisfies  (i)-(v) above. 

Below, we will define a geometric notion of 
genericity where the nice types are given by the special types. Let us first show that there 
are `enough' special types. 
Let $K\models T_{\omega}$ and  $\a$ be an arbitrary finite tuple from $\mathfrak{C}\elres K$. Choose some finite green tuple $\overline{u}$ such that $\a\in K[\overline{u}]$. Such a tuple exists by Fact \ref{F:PoizatGreen}(f).

Combining the fact that $\tr(\cl_{\omega}(K\overline{u})/K)$ is finite with Fact \ref{F:FreeAb}, we may find some finite tuple $\tilde{a}$ containing $\a$, 
$\tilde{a}=\overline{g}\overline{g'}\overline{w}\overline{w'}$ (where all the elements outside $\a$ 
may be taken to be green) 
such that $\tp(\tilde{a}/K)$ is special. 

\medskip
Now suppose $\tp(\tilde{a}/K)$ is special, with 
$\tilde{a}=\overline{g}\overline{g'}\overline{w}\overline{w'}$ as above. Choose a finite 
$\b\in K$ such that 
\begin{itemize}
\item $\Loc(\overline{g}/K)=U(\x,\b)$ is defined over $\b$, similarly 
$\Loc(\overline{g},\overline{w}/K)=V(\x,\y,\b)$ and $\Loc(\tilde{a}/K)=W(\x,\x',\y,\y',\b)$. (This is equivalent to $\Cb_{\mathrm{ACF}}(\tilde{a}/K)\subseteq\b$.)
\item For any $g'$ from $\overline{g'}$ there exists a green $b_{g'}$ from $\b$ and $m_{1,g'},\ldots,m_{n,g'}\in\Z$ such that 
\begin{equation}
g'=b_{g'}\prod_{i=1}^n g_i^{m_{i,g'}}.\label{Eq:green}
\end{equation}
\item For any $w'$ from 
$\overline{w'}$ there exists some $b_{w'}$ from $\b$, integers $m_{i,w'}$, $1\leq i\leq n$ and 
$n_{i,w'}$, $1\leq i\leq l$
such that 
\begin{equation}
w'=b_{w'}\prod_{i=1}^n g_i^{m_{i,g'}}\prod_{i=1}^l w_i^{n_{i,w'}}.
\label{Eq:white}
\end{equation}
Moreover, if $n_{i,w'}=0$ for all $i$, then $b_{w'}$ is a white element.
\item For any $w$ from 
$\overline{w}$ there is a polynomial $f_w\in K[\overline{u},\overline{v}]$ with coefficients from $\b$ (so we may write it $f_w=F_w(\overline{u},\overline{v},\b)$) such that
\begin{equation}
w=F_w(g_1,\ldots,g_n,\frac{1}{g_1},\ldots,\frac{1}{g_n},\b).
\label{Eq:gay}
\end{equation}
 
\end{itemize}

Let $k:=\dim(U)=\dim(W)=\tr(\overline{g}/K)$ and $d:=\d(\tilde{a}/K)=2k-n$. The following conditions 
(\ref{Anfang}-\ref{Colours}) hold for $\b'=\b$, and they are definable in $\b'$. (Definability follows from 
Fact \ref{F:RotundDef} and Proposition \ref{P:SimpleDef}.)

\begin{enumerate}[\upshape(a)]\setcounter{enumi}{\value{Gen_enum}} 
\item\label{Anfang} $U(\x,\b')$ is irreducible of dimension $k$.
\item\label{Rotund} $U(\x,\b')$ is rotund.
\item \label{Kgen} $U(\x,\b')$ is Kummer generic.
\item \label{Bunt} $V(\x,\y,\b')$ is equal to the variety given by $U(\x,\b')$ together with the equations from \eqref{Eq:gay}, 
i.e.\ $y_w=F_w(x_1,\ldots,x_n,\frac{1}{x_1},\ldots,\frac{1}{x_n},\b')$, where $y_w$ is the variable 
corresponding to $w$.
\item \label{WholeBasis} $V(\x,\y,\b')\subseteq\Gm^{n+l}$ is free (where $l=\lgth(\y)$).
\item \label{GreenWhite} $W(\x,\x',\y,\y',\b')$ is equal to the variety given by $V(\x,\y,\b')$ together with the the corresponding equations from \eqref{Eq:green} and \eqref{Eq:white}.
\item \label{Colours} The quantifier-free types of $\b$ and $\b'$ in the language $\{\UE,=\}$ coincide.
\end{enumerate}

Let $\theta(\z)$ be an $\L$-formula such that $\models\theta(\b')$ if and only if the conditions (\ref{Anfang}-\ref{Colours}) are satisfied.

\medskip

\noindent
Claim. \emph{Suppose that $\b'\in K\models T_{\omega}$ such that $\models\theta(\b')$. Let 
$\tilde{a}=\overline{g}\overline{g'}\overline{w}\overline{w'}$ be a $K$-generic solution of $W_{\b'}$, 
and let $L:=K(\tilde{a})^{alg}=K(\overline{g})^{alg}$, $\UE(L):=\divh(\langle\UE(K)\overline{g}\rangle)$. 
Then $(L,\UE(L))$ is a self-sufficient extension of $(K,\UE(K))$, with $\delta(L/K)=d$. The tuple 
$\overline{g}\overline{g'}$ 
consists of green elements, whereas the elements from $\overline{w}\overline{w'}$ are white. 
\newline
Assume in addition that $L$ is self-sufficient in $\mathfrak{C}$. Then, $\tp_{\L}(\tilde{a}/K)$ is uniquely determined by: $\tilde{a}$ is (field) generic in $W_{\b'}$ over $K$, $\overline{g}\overline{g'}$ is green, $\overline{w}\overline{w'}$ is white and $K\tilde{a}\leq \mathfrak{C}$.} 

\medskip

By construction, $K\tilde{a}\leq L$ (and so also $K\leq L$) follows from (\ref{Rotund}). 
The fact that $\overline{g}\overline{g'}$ is a green tuple is true by construction, together with 
(\ref{GreenWhite}) and (\ref{Colours}). The colour assigned to each 
element $w$ of $\overline{w}$ is white, since this is a multiplicatively independent tuple 
over $K\overline{g}$ by (\ref{WholeBasis}) and (\ref{Bunt}). Combining (\ref{GreenWhite}) and 
(\ref{Colours}), we see that $\overline{w}'$ consists of white elements only. Note that the irreducibility of $W_{\b'}$ as well as $\delta(L/K)=d$ is an easy consequence of 
(\ref{Anfang}) together with the other conditions. 

Finally, if $L\leq \mathfrak{C}$ (from which we deduce $K\tilde{a}\leq\mathfrak{C}$), the 
type of $\tilde{a}$ over $K$ is determined in the described way, 
since $U_{\b'}$ is Kummer generic (and $W_{\b'}$ irreducible). This follows from Corollary \ref{C:KummerSameType} and 
proves the claim.

\medskip

Assume that $U,V,W,\theta,d$ are given as before, in the variables $\tilde{x},\z$, where 
$\tilde{x}= \x\x'\y\y'$. 
\begin{itemize}
\item Let $\phi_{\UE}(\tilde{x})$ be a formula expressing that the elements from $\x,\x'$ are green, 
and those from $\y,\y'$ white;
\item let $\phi_d(\tilde{x},\z)$ be an arbitrary formula from the partial type $\pi_d(\tilde{x},\z)$ (introduced in Fact 
\ref{F:DimTypeDef});
\item let $Z(\tilde{x},\z)$ be a uniform family of varieties;
\item let $\chi_1(\z)$ be a formula such that $\models \chi_1(\b')$ if and only if $Z_{\b'}$ is a proper 
subvariety of $W_{\b'}$;
\item let $\chi(\z)=\chi_1(\z)\wedge\theta(\z)$.
\end{itemize}

A \emph{special} formula is a formula of the form
\begin{equation}\label{Eq:Special}
\phi(\tilde{x},\z)=W(\tilde{x},\z)\wedge\neg Z(\tilde{x},\z)\wedge\phi_{\UE}(\tilde{x})\wedge\phi_d(\tilde{x},\z)\wedge\chi(\z).
\end{equation}

A non-empty instance of a special formula will also be called special. 

By the claim, for any $\b'\in K\models T_{\omega}$ such that $\models\chi(\b')$ there is a unique 
special type $p(\tilde{x})\in S(K)$ containing $\phi(\tilde{x},\b')$ such that any realisation 
$\tilde{a}$ of $p$ is generic (in the field sense) in $W_{\b'}$ over $K$. Moreover, using 
Fact \ref{F:DimTypeDef}, it is easy to see that the set of (instances of) special formulas in a 
given special type is dense in it.

\smallskip

We now define a notion of genericity, only using special formulas and special types. Let 
$\phi(\tilde{x},\b')$ and $p\in S(K)$ be special, and assume that $p$ contains $\phi(\tilde{x},\b')$. 
We say that $p$ is \emph{generic} in $\phi(\tilde{x},\b')$ if any $\tilde{a}\models  p$ is 
field generic in $W_{\b'}$, where $W(\tilde{x},\b')$ is as in \eqref{Eq:Special}. 

\smallskip

By what we have seen, special formulas (types, resp.) correspond to nice formulas (types, resp.). 
Since there are enough special types, in order to show that the notion of genericity we defined is a geometric 
notion of genericity, it is sufficient to show that it satisfies property (4) from Definition \ref{D:GeometricGenericity} (the remaining properties are clear).

\smallskip

To prove property (4), assume that $p$ is generic in $\phi(\tilde{x},\b)$, $p_0$ generic 
in $\phi_0(\tilde{x}_0,\b_0)$, $\tilde{x}_0$ is a subtuple of $\tilde{x}$, and that 
$p$ restricts to $p_0$. We have to find $\delta(\z,\z_0)\in\tp(\b\b_0)$ such that 
whenever $\models\delta(\b',\b_0')$, the generic type of $\phi(\tilde{x},\b')$ restricts 
to the generic type of $\phi_0(\tilde{x}_0,\b_0)$. Choose $\tilde{a}\models p$. Then 
$\tilde{a}_0=\overline{g}_0\overline{g}'_0\overline{w}_0\overline{w}'_0\models 
p_0$ and we observe: 
\begin{eqnarray}
\label{Eq:RelRotund}
K\tilde{a}_0\leq K(\tilde{a})^{alg}=:L\text{  or, equivalently,  }K\overline{g}_0\leq L\\
\label{Eq:FieldGen}
W(\tilde{x},\b) \text{  projects onto a generic subset of  }W_0(\tilde{x}_0,\b_0).
\end{eqnarray} 

Extend $\overline{g}_0$ to a (green) basis $\overline{g}_0\overline{g}_1\subseteq \overline{g}\overline{g}'$ 
of $\UE(L)$ over $\UE(K)$, and let $\x_1$ be the variable tuple corresponding to $\overline{g}_1$. Choose a formula 
$\delta(\z,\z_0)\in\tp(\b,\b_0)$ such that for any $(\b',\b_0')$ with $\models\delta(\b',\b_0')$ the following three conditions hold:
\begin{enumerate}
\item $\models \chi(\b')\wedge\chi_0(\b'_0))$.
\item $W(\tilde{x},\b')$ projects onto a generic subset of $W_0(\tilde{x}_0,\b'_0)$.
\item For generic 
$\tilde{a}=\overline{g}\overline{g}'\overline{w}\overline{w}'$ in $W_{\b'}$, the variety 
$\Loc(\overline{g}_1/\Q(\b'\b'_0\overline{g}_0)^{alg})$ is rotund.
\end{enumerate}
Note that the last property can be guaranteed using definability of types in algebraically 
closed fields, combined with Fact \ref{F:RotundDef}.

Let $\b',\b_0'\in K'\models T_{\omega}$ such that $K'\models\delta(\b',\b_0')$. By the above 
conditions on $\delta$, the generic type $p'(\tilde{x})\in S(K')$ of the special formula $\phi(\tilde{x},\b')$ 
restricts to the generic type $p_0'(\tilde{x}_0)$ of $\phi_0(\tilde{x}_0,\b_0')$. This is clear for 
the algebraic part of the type as for the colouring. Moreover, if $\tilde{a}\models p'$, then 
$K\tilde{a}_0\leq L=K(\tilde{a})^{alg}$ follows from (3) and Lemma \ref{L:RotundSelfsuff}. Since 
$L\leq\mathfrak{C}$, we deduce that $K\tilde{a}_0\leq\mathfrak{C}$, and so $\tilde{a}_0\models p'_0$.
\end{proof}

\begin{Thm}\label{T:GenGreen}
Let $T_{\omega}$ be the theory of the green field of Poizat (considered in an expansion by definition 
so that it eliminates quantifiers). Then $T_{\omega}A$ exists. Its reduct to the language of 
difference fields is equal to $ACFA_0$.
\end{Thm}

\begin{proof}
Lemma \ref{L:GreenGenNotion} shows that $T_{\omega}$ admits a geometric notion of genericity. 
Thus $T_{\omega}A$ exists by Proposition \ref{P:PartialGenTAExists}.

\smallskip

Now consider $\left(K,\UE(K),\sigma\right)\models T_{\omega}A$. Suppose that 
$(K,\sigma)\subseteq(L,\sigma)\models ACFA_0$. Putting $\UE(L):=\UE(K)$, we may 
expand the difference field $(L,\sigma)$ to a green field with automorphism $(L,\UE(L),\sigma)$. 
Then $K\leq L$, and there is $\left(L,\UE(L),\sigma\right)\subseteq\left(M,\UE(M),\sigma\right)\models T_{\omega}A$ 
such that $L\leq M$. Since $(K,\UE(K),\sigma)\elex (M,\UE(M),\sigma)$, it follows in particular 
that $(K,\sigma)$ is existentially closed in $(L,\sigma)$. Thus, $(K,\sigma)$ is an existentially closed difference field, i.e.\ a model of $ACFA_0$. 

Let us now show that every completion of $ACFA_0$ is attained in this manner. Note that for any 
$\sigma\in\Gal(\Q)$, the green field with automorphism $(\Q^{alg},\{1\},\sigma)$ embeds (in a 
self-sufficient way) into a model of $T_{\omega}A$. By Fact \ref{T:CP}(2), any completion of $ACFA_0$ is determined 
by the action of $\sigma$ on $\Q^{alg}$. This shows the result.
\end{proof}

We already mentioned that for stable $T$, the existence of $TA$ implies that $T$ does not have the finite cover 
property. Thus, Theorem \ref{T:GenGreen} implies the following result of Evans \cite{Ev08}.

\begin{Cor}
The green field of Poizat does not have the finite cover property.\qed
\end{Cor}

\subsection{Generic automorphisms of bad fields}

\begin{Thm}\label{T:Bad_DMP}
The theory $T_{\mu}$ has DMP.
\end{Thm}

\begin{proof}

Since Morley rank is finite, definable 
and additive in $T_{\mu}$, to show DMP, it is sufficient to find, for any type $p\in S_m(M)$ over a model 
$M$, a formula $\phi(\x,\b)\in p$ such that 
$\RDM(p)=\RDM(\phi(\x,\b))=(d,1)$ and $\RDM(\phi(\x,\b')=(d,1)$ whenever $\phi(\x,\b')$ 
is consistent. Call such a type $p$ \emph{good}.

\smallskip
\noindent
Claim. \emph{Suppose that $q(\tilde{x})\in S_m(M)$ is a good type which is a finite cover of $p(\x)\in S_n(M)$, i.e.\ 
there is a partial $M$-definable function $f$ with finite fibers such that $f_*(q)=p$. Then, $p$ is good.}
(The proof is left to the reader.)

\smallskip



Now, let $p=\tp(\b/M)\in S_m(M)$, for $M\models T_{\mu}$. Note that there is a finite green tuple $\a'$ which is 
algebraic (in the sense of $T_{\mu}$) over $M\b$ and such that $\b\in\dcl_{\mu}(M\a')$ (since $M':=\acl_{\mu}(M\b)\elres M$ and $\dcl_{\mu}(\UE(M'))=M'$ by Fact \ref{F:Bad_IE} and Fact \ref{F:BadField}). We may even assume that $\a'$ is a green basis 
of a self-sufficient extension of $M$. By Fact \ref{F:FreeAb}, 
there are elements $a_1,\ldots,a_n\in M(\a')$, the field generated by $\a'$ over $M$, such that 
$M^{\times}\langle \a\rangle=\divh(M^{\times}\langle\a'\rangle)\cap M(\a')$, i.e.\ $\a$ is Kummer generic over $K$. 
Replacing $a_i$ by $\zeta(i) a_i$ for some root of unity $\zeta(i)$, we may arrange that $a_i$ is green 
for all $i$.  

We still have $\a\in\acl_{\mu}(M\b)$ and $\b\in\dcl_{\mu}(M\a)$. So by the claim, it is sufficient to show 
that $q=\tp(\a/M)$ is good. Note that $V=\Loc(\a/M)$ is a Kummer generic variety. Thus, by Corollary 
\ref{C:KummerSameType}, one has $\a_1\models q$ 
if and only if the following conditions hold:

\begin{itemize}
\item $\a_1$ is generic (in the field sense) in $V$ over $M$, i.e.\ $|Loc(\a_1/M)=V$;
\item the tuple $\a_1$ consists of green elements;
\item $\d(\a_1/M)=\delta(\a_1/M)=2\dim(V)-n=\delta(V)=d$.
\end{itemize}

Suppose $\RDM\left(V(\x)\wedge\bigwedge_{i=1}^n\UE(x_i)\right)=(d',m')>(d,1)$. Then there is a proper 
subvariety $W$ of $V$ containing a type of maximal Morley rank $d'$. Since $\RM=\d\leq\delta$, $W$ is contained in a coset of some $T\in\mathcal{T}(V)$, $T\neq\Gm^n$, by Lemma \ref{L:MaxDelta}. It will suffice to remove from $V$ (performing an induction) a finite number of such cosets to get a definable set of 
$\RDM$ equal to $(d,1)$, for at each 
such step, either the Morley rank or the Morley degree will drop. After a finite number of steps we 
thus arrive at a formula of the form 
$$\phi(\x,\b,\overline{c})=\bigwedge_{i=1}^n\UE(x_i)\wedge\x\in V_{\b}\setminus\left(\bigcup_{\Gm^n\neq T\in\mathcal{T}}\bigcup_{i=1}^{n_T}\overline{c}_{T,i}\cdot T\right)$$
such that $\RM(\phi(\x,\b,\overline{c}))=(d,1)$ with $q$ as its unique generic type.

The following are definable conditions in the parameters $\b',\overline{c}'$ 
(by Proposition \ref{P:SimpleDef} and Fact \ref{F:BadField}, since Morley rank is definable in an almost strongly minimal theory):

\begin{enumerate}
\item[(*)] $V_{\b'}$ is Kummer generic and $\delta(V_{b'})=d$;
 \item[(**)] $\RM(\phi(\x,\b',\overline{c}'))=d$;
 \item[(***)] for any $T\in\mathcal{T}(V)$ such that $T\neq\Gm^n$, the intersection of $\phi(\x,\b',\overline{c}'))$ 
 with any coset of $T$ is of Morley rank $<d$.
 \end{enumerate}
 
If $(*),(**)$ and $(***)$ are satisfied, then $\RDM(\phi(\x,\b',\overline{c}'))=(d,1)$, showing that 
$q$ is a good type.
\end{proof}

Note that in the previous proof, the conditions $(*),(**)$ and $(***)$ guarantee that when assigning 
the green colour to a generic point $\a'$ of $V_{\b'}$ (over $K'\models T_{\mu}$), we obtain a 
self-sufficient extension of $K'$ which stays in the class $\Cmu$. A priori, it is 
not clear that this is a definable condition in the parameters.  

\begin{Thm}\label{T:Gen_Bad_Field}
The theory $T_{\mu}A$ exists. The $\L_{rings}\cup\{\sigma\}$-reduct of $T_{\mu}A$ 
equals $ACFA_0$.
\end{Thm}
\begin{proof}
The existence of $T_{\mu}A$  follows from Theorem \ref{T:Bad_DMP}, using \ref{Ex:AddDMP}.

\smallskip

Note that if $(K,\UE(K))$ is a green field from the class $\Cmu$ and $L$ is an 
algebraically closed field containing $K$, then $(L,\UE(K))\in\Cmu$ 
(see \cite[Folgerung 8.3]{BHMW07}). Using this, the argument concerning the reduct to the language of difference fields is the same as in the proof of Theorem \ref{T:GenGreen}.  
\end{proof}
 
 \subsection{Bad pseudofinite fields}
 
 We now give an application to pseudofinite fields, showing that the fixed field of a model of $T_{\mu}A$ is what 
 might be called a `bad pseudofinite field' of characteristic 0.
 
Recall that every pseudofinite field is supersimple of $\SU$-rank 1, with $\SU(\a/K)=\tr(\a/K)$ (see \cite{Hr02} and \cite{Wa00} for pseudo-finite fields and simple 
theories).  

\begin{Cor} \label{C:Psf_Bad}
Let $F'$ be a pseudofinite field of characteristic 0. Then, there is $F\elres F'$ and an infinite 
divisible torsion free subgroup $\UE$ of the multiplicative group of $F$ such that $(F,+,\times,\UE)$ is supersimple of 
$\SU$-rank $2$, with $\UE$ of $\SU$-rank $1$.
\end{Cor}

\begin{proof}
Choose $K\models T_{\mu}A$ (sufficiently saturated) such that the fixed field $F$ is an elementary extension 
of $F'$. (It is easy to see that there is $(K,\sigma)\models ACFA_0$ with this property \cite{CH99}; by Theorem 
\ref{T:Gen_Bad_Field}, this is sufficient.) Note that $F$ is stably embedded, by Fact \ref{T:CP}(4) and Fact \ref{F:Bad_IE}. We first show that the full induced structure on $F$ is supersimple of 
$\SU$-rank 2, with $\SU(\UE)=1$. We denote this theory by $\mathrm{Th}(F)$. 
Choose an element $g\in \UE(F)$. Note that $\acl_{\mathrm{Th}(F)}(B)=\acl_{T_{\mu}A}(B)\cap F
=\acl_{\mu}(B)\cap F$ for any $B\subseteq F$ (by Fact \ref{T:CP}(1)). Now assume $g\not\in\acl_{\mu}(\emptyset)$. Then $\d(g)=\RM_{T_{\mu}}(g/\emptyset)=1$. 
For every $B\subseteq F$, the following holds:
$$g\ind^{T_{\mu}}B\Leftrightarrow g\not\in\acl_{\mu}(B)\Leftrightarrow g\not\in\acl_{\mathrm{Th}(F)}(B)\Leftrightarrow g\ind^{\mathrm{Th}(F)}B.$$
So the only forking extensions of $\tp_{\mathrm{Th}(F)}(g)$ are algebraic, from which we 
deduce $\SU_{\mathrm{Th}(F)}(g)=1$, thus $\SU_{\mathrm{Th}(F)}(\UE)=1$. Next we show that there 
is a 1-type in $\mathrm{Th}(F)$ of rank 2. Choose generic 
independent green elements $g_1,g_2$ in $F$, and put $w=g_1+g_2$. Then $\delta(g_1,g_2/w)=0$, so $w$ and $(g_1,g_2)$ are interalgebraic (in $T_{\mu}$, so also in $\mathrm{Th}(F)$). 
By the Lascar inequalities, we compute $\SU_{\mathrm{Th}(F)}(w)=\SU_{\mathrm{Th}(F)}(g_1,g_2)=2$.

On the other hand, using the characterisation of non-forking in Fact \ref{T:CP}(3), by an easy 
induction on Morley rank we show that for any $\a\in F$ and $B\subseteq F$ one has 
$\RM_{T_{\mu}}(\a/B)\geq\SU_{\mathrm{Th}(F)}(\a/B)$.
 
The structure $F_{\UE}=(F,+,\times,\UE)$ is a reduct of the full induced structure on $F$. Moreover, as 
$\UE$ is an infinite definable subgroup of infinite index in the multiplicative group of $F$, it 
follows that $\SU(F_{\UE})\geq2$. We finish the proof using the following general lemma (it is folklore; for convenience, we include a proof).
\end{proof}

\begin{Lem}
Let $T'$ be a reduct of the simple theory $T$, and let $\pi'$ be a partial $T'$-type such that 
$\SU_T(\pi')<\omega$. Then $\SU_{T'}(\pi')\leq\SU_T(\pi')$.
\end{Lem}

\begin{proof}
We argue by induction on $\SU_{T'}(\pi')=n\in\N$, the case $n=0$ being trivial. Suppose 
$\SU_{T'}(\pi')=n+1$, where $\pi'$ is a partial type over $A$. Taking a $T'$-non-forking extension if 
necessary, we may suppose $A=M\models T$. Let $B\supset M$ and $a\models\pi'$ such 
that $\SU_{T'}(a/B)=n$. By induction, we know that $\SU_T(a/B)\geq n$. Moreover, 
since $a\nind^{T'}_M B$ is witnessed by any $T'$-Morley sequence in $\tp_{T'}(B/M)$ (see e.g.\ \cite[Thm 2.4.7]{Wa00}), 
in particular by any $T$-coheir sequence in $\tp_T(B/M)$, we may deduce $a\nind^{T}_M B$ from 
$a\nind^{T'}_M B$.
\end{proof}

\begin{Pb}
Is it possible to obtain the green pseudofinite field $(F,\UE)$ of rank 2 (or the one of infinite rank) as an ultraproduct of coloured finite fields $(\F_q,N_q)$? 
\end{Pb}

\section{Other Hrushovski amalgams}\label{S:Other}
We briefly review Hrushovski's amalgamation method (see \cite{HH06} for a detailed account of this method). It is a variation of Fra\"{i}ss\'e's original 
method, and a powerful tool to construct stable structures with prescribed pregeometry. 

Let $\mathcal{C}$ be a class of $\L$-structures, $\mathcal{C}^{fin}\subseteq\mathcal{C}$ the class of finite (or `finitely generated' in some sense) structures in $\mathcal{C}$, and 
let  $\delta:\mathcal{C}^{fin}\rightarrow\Z$ be a \emph{predimension function} satisfying some natural conditions. 
For $A\subseteq B$ in $\mathcal{C}^{fin}$ put 
$\delta(B/A):=\delta(B)-\delta(A)$ (this definition may be extended to infinite $A$, as long as $B$ is finitely generated over 
$A$). The structure $A$ is said to be \emph{self-sufficient} in $B$ (denoted by $A\leq B$) if $\delta(B'/A)\geq0$ for any 
$A\subseteq B'\subseteq B$ with $B'$ finitely generated over $A$. Let $\C=\{M\in\C\,\mid\,\emptyset\leq M\}$ and consider the class $(\C,\leq)$. In all examples 
we treat, $\C$ is an elementary class, $\C^{fin}$ is countable up to isomorphism, and $(\C,\leq)$ has AP and JEP. 
So there is a unique countable structure $M_{\omega}$ in $\C$ which is homogeneous with respect to $(\C,\leq)$, the 
\emph{Fra\"{i}ss\'e-Hrushovski limit} 
of $(\C^{fin},\leq)$. In order to establish the desired properties for 
$T_{\omega}=\Th(M_{\omega})$ one has to show that $M_{\omega}$ is saturated.

The theory $T_{\omega}$ obtained in this way is usually of infinite (Morley) rank, and a more intricate second step --- the 
so-called \emph{collapse} --- is needed to obtain a theory of finite Morley rank where the rank is given by the \emph{dimension} (i.e.\ the `eventual predimension') that comes out of the construction,
$\d(A):=\min\{\delta(A')\,\mid\,A\subseteq A'\subseteq K\}$. The rough idea is to bound uniformly the number of realisations of 
types in $T_{\omega}$ of dimension 0. Technically, this is done by choosing families of strongly minimal sets in 
$T_{\omega}$ which coordinatise all such types of dimension 0 and to associate to any such family $\mathcal{F}$ a natural number $\mu(\mathcal{F})$. One obtains an elementary subclass $\C^{\mu}\subseteq\C$. The most delicate 
parts of the construction are to establish that $(\C^{\mu},\leq)$ has AP, and the fact that the Fra\"{i}ss\'e-Hrushovski limit 
$M_{\mu}$ of  the finite structures in $(\C^{\mu},\leq)$ is saturated. All this is analogous to the construction of green and 
bad fields which was outlined in Section \ref{S:delta}. 

A famous instance of the aforementioned amalgamation method is Hrushovski's \emph{fusion construction}, where 
two arbitrary strongly minimal theories (with DMP) are fused into a single strongly minimal theory (\cite{Hr92}, see also 
\cite{HH06} for a detailed exposition of the uncollapsed fusion). For $i=1,2$, let $T_i$ be strongly minimal $\L_i$-theories 
with DMP. We may assume that $\L_1$ and $\L_2$ are disjoint relational languages and that $T_i$ has quantifier elimination. 
For $\L:=\L_1\cup\L_2$, consider the class $\mathcal{C}$ of models of the $\L$-theory $T_1^{\forall}\cup T_2^{\forall}$, and 
for finite $A\in\mathcal{C}$ put $\delta(A)=\d_1(A)+\d_2(A)-\!\mid \!A\!\mid$, where $\d_i(A)$ is Morley rank in the sense of $T_i$. The 
above techniques apply. The theory $T_{\omega}$ is called the \emph{free fusion} of $T_1$ and $T_2$ (over equality); the 
desired strongly minimal fusion is given by $T_{\mu}$.

\begin{Fact}\label{F:FusionUseful}
Let $T_{\omega}$ be the free fusion of the strongly minimal theories $T_1$ and $T_2$. 

\begin{enumerate}
\item $T_{\omega}$ is $\omega$-stable.
\item Let $A$ and $A'$ be self-sufficient subsets of $\mathfrak{C}\models T_{\omega}$. Then $\tp_{\omega}(A)=\tp_{\omega}(A')$ if and 
only if $\tp_{T_i}(A)=\tp_{T_i}(A')$ for $i=1,2$.
\item Let $K\elex\mathfrak{C}$ and $\a\in\mathfrak{C}$ be a finite tuple. There is some finite $\hat{a}\supseteq\a$ 
such that $K\hat{a}\leq\mathfrak{C}$.
\item Let $\x=(x_0,\ldots,x_{n-1})$, $0\leq d\leq n$ and $\z$ an arbitrary tuple of variables. Then there is a partial type 
$\pi_d(\x,\z)$ such that for any $\a,\b$ one has $\models\pi_d(\a,\b)$ if and only if $\d(\a/\b)\geq d$.
\end{enumerate}
\end{Fact}

\begin{proof}
The first three items are proved in \cite{HH06}, and the last part is an easy consequence of definability 
of Morley rank in strongly minimal theories.
\end{proof}

\begin{Thm}\label{P:OtherAmalgams}
For the following theories obtained by Hrushovski's amalgamation method without collapse, all $\omega$-stable 
of infinite rank, the generic automorphism is axiomatisable:  
\begin{enumerate}
\item The \emph{ab initio} construction \cite{Hr93}.
\item The \emph{free fusion} of two strongly minimal theories $T_1$ and $T_2$ , where both $T_1$ and $T_2$ have DMP 
\cite{Hr92} (see also \cite{HH06}). 
\item The free fusion of two strongly minimal theories $T_1$ and $T_2$ \emph{over a common subtheory} $T_0$, where both $T_1$ and $T_2$ have DMP and $T_0$ is $\omega$-categorical, modular and satisfies $\acl_{T_0}=\dcl_{T_0}$ (e.g. for 
$T_0$ the theory of an infinite vector space over some finite field) \cite{HH06}. 
\item The black fields of Poizat in all characteristics \cite{Po99}.
\item The red fields of Poizat in positive characteristic \cite{Po01}.
\item The theory of the generic plane curve over an algebraically closed field constructed in \cite{CHKP02}).
\end{enumerate}
\end{Thm}

\begin{proof}
We give the argument for (2), the other cases being similar. So let $T_{\omega}$ be the free fusion of two strongly 
minimal theories $T_1$ and $T_2$ having DMP. We will exhibit a geometric notion of genericity and apply Proposition 
\ref{P:PartialGenTAExists}. The construction is parallel to the one given in Lemma \ref{L:GreenGenNotion}, although the 
definabiliy problems we encountered in the case of Poizat's green fields do not arise in the context of the free fusion.

\smallskip

Let $K\elex\mathfrak{C}$ and let $\a\in\mathfrak{C}$ be a finite tuple. Then $\tp_{\omega}(\a/K)$ is called \emph{special} if 
$K\a\leq\mathfrak{C}$. Now let $p(\x)=\tp_{\omega}(\a/K)$ be special. For convenience we will assume that $\a=(a_0,\ldots,a_{n-1})$ enumerates 
$K\a\setminus K$ (without repetitions). By Fact \ref{F:FusionUseful}, $p$ is determined by $p_1=\tp_1(\a/K)$ and 
$p_2=\tp_2(\a/K)$. 
For $I\subseteq \{0,\ldots, n-1\}={\bf n}$, let $k^I_i:=\RM_{T_i}(a_I/K)$. Then (by the assumptions) the following constraints are satisfied:
\begin{eqnarray}
\label{Eq:nonalg}
k^I_1>0\text{ and }k^{I}_2>0\text{ whenever }I\neq\emptyset,\\
\label{Eq:nonneg}
k^{I}_1+k^{I}_2-\!\mid\! I\!\mid\geq0\text{ for all }I\subseteq{\bf n},\\
\label{Eq:DimPredim}
d(\a/K)=\delta(\a/K)=k^{\bf n}_1+k^{\bf n}_2-n.
\end{eqnarray}

We choose $\L_i$-formuals $\phi_i(\x,\z_i)$ and $\b_i\in K$ such that
\begin{itemize} 
 \item[(i)] $p_i$ is the unique $\L_i$-generic type in $\phi_i(\x,\b_i)$ over $K$ (for $i=1,2$),
 \item[(ii)] the formulas $\phi_i(\x,\z_i)$ avoid all diagonals,
\item[(iii)] if $\phi_i(\x,\b')\neq\emptyset$, then this is a formula of Morley degree 1,
\item[(iv)] if $\a'$ is $\L_i$-generic in $\phi_i(\x,\b_i')$ over $K'$ (where $\b_i'\in K'$), then 
$\RM_{T_i}(a'_I/K')=k^I_i$ for all $I\subseteq{\bf n}$.
\end{itemize}
Using DMP in $T_i$, it is easy to see that formulas of the form $\phi_i(\x,\b_i)$ exist and are dense in $p_i$.

Put $d=d(\a/K)=k^{\bf n}_1+k^{\bf n}_2-n$. A \emph{special} formula is a formula of the form 
$$\phi(\x,\z)=\phi_1(\x,\z_1)\wedge\phi_2(\x,\z_2)\wedge\phi_d(\x,\z),$$
where $\z\supseteq\z_1,\z_2$ is some tuple of variables, $\phi_d(\x,\z)$ a formula from $\pi_d(\x,\z)$ 
(see Fact \ref{F:FusionUseful}) and $\phi_1,\phi_2$ are as above, satisfying (i-iv). 

If $p(\x)=\tp_{\omega}(\a/K)$ is special with $\d(\a/K)=d$, then 
$p_1(\x)\cup p_2(\x)\cup\pi_d(\x,K)\vdash p(\x)$, where $p_i=p\res_{\L_i}$. It follows that (instances of) special formulas 
are dense in any special type.

\smallskip
\noindent
Claim. \emph{Let $\phi(\x,\z)=\phi_1(\x,\z_1)\wedge\phi_2(\x,\z_2)\wedge\phi_d(\x,\z)$ be a special formula and 
$\b'\in K'\models T_{\omega}$ such that 
$\exists\x\phi_1(\x,\b'_1)\wedge\exists\x\phi_2(\x,\b'_2)$. Then there is a unique 
special type $p\in S_n(K')$ such that $p_i=p\res_{\L_{i}}$ is generic in $\phi_i(\x,\b_i')$ for $i=1,2$.}
\smallskip

Let $p_i$ be the (by (iii) unique) generic $\L_i$-type in $\phi_i(\x,\b_i')$. Note that by (ii) the types $p_1$ and $p_2$ agree 
on their reduct to mere equality. By \cite[Lemma 3.10]{HH06} and \eqref{Eq:nonneg} 
there is $\a'\models p_1\cup p_2$ such that $K'\a'\leq\mathfrak{C}$. Then $\tp(\a'/K')$ is special and uniquely determined 
by these data by Fact \ref{F:FusionUseful}(2). From $K\a'\leq\mathfrak{C}$ we deduce that $\d(\a'/K')=\delta(\a'/K')= k^{\bf n}_1+k^{\bf n}_2-n=d$, so $\models\pi_d(\a',\b')$ and in particular $\models\phi(\a',\b')$. This proves the claim.

\smallskip

We now define a notion of genericity $R_{\g}$ for special types and formulas. We say that the special type $p(\x)\in S(K)$ is 
\emph{generic} in the special formula $\phi(\x,\b)=\phi_1(\x,\b_1)\wedge\phi_2(\x,\b_2)\wedge\phi_d(\x,\b)$ if 
$p_i=p\res_{\L_i}$ is generic in $\phi_i(\x,\b_i)$ for $i=1,2$. 

It follows from the claim that the special types / formulas 
are precisely the nice types / formulas with respect to $R_{\g}$. Clearly, $R_{\g}$ is invariant and coherent. Moreover, by Fact 
\ref{F:FusionUseful}(3), there are enough nice types. In order to prove that $R_{\g}$ is geometric, it remains to show property 
(4) from Definition \ref{D:GeometricGenericity}, i.e.\ the definability of generic projections.

Let $\x=\x'{\ensuremath{\mathaccent\cdot\cup}}\x''$, $\x'=x_{I'}$ for $I'\subseteq{\bf n}$, and let $\phi'(\x',\z')=\phi'_1(\x',\z'_1)\wedge\phi'_2(\x',\z'_2)\wedge\phi'_{d'}(\x',\z')$ and $\phi(\x,\z)=\phi_1(\x,\z_1)\wedge\phi_2(\x,\z_2)\wedge\phi_d(\x,\z)$ be special formulas, where the integers $k^{I}_i$ are associated with $\phi(\x,\z)$). Let $\b,\b'\inÊK\models 
T_{\omega}$ be such that both $\phi(\x,\b)$ and $\phi'(\x',\b')$ are non-empty. Clearly, the projection of $\phi(\x,\b)$ onto 
the $\x'$-variables is generic in $\phi'(\x',,\b')$ if and only if for any generic (over $K$) $\a\models\phi(\x,\b)$, the tuple 
$\a'=a_{I'}$ is generic in $\phi'(\x',\b')$ over $K$, i.e.\ if $\a'$ is $\L_i$-generic in $\phi'_i(\x',\b_i')$ over $K$ for $i=1,2$ and 
$K\a'\leq\mathfrak{C}$, or equivalently $K\a'\leq K\a$. This is the case if and only if the following two definable properties 
hold (note that the second one is either always or never satisfied for a given pair of special formulas):
\begin{itemize}
\item $\exists \x''\phi_i(\x'\x'',\b_i)$ is $\L_i$-generic in $\phi_i'(\x',\b_i')$ for $i=1,2$;
\item $k_1^J+k_2^J-\!\mid\! J\!\mid \geq k_1^{I'}+k_2^{I'}-\!\mid \!I'\!\mid$ for all $I'\subseteq J\subseteq{\bf n}$.\hfill\qedhere
\end{itemize}
\end{proof}

It is known that all the theories from Theorem \ref{P:OtherAmalgams} may be collapsed onto theories of finite rank. Using arguments which 
are similar (albeit much simpler) to the proof of Theorem \ref{T:Bad_DMP}, one obtains the following result.

\begin{Thm}
The collapsed versions of all the theories from Theorem \ref{P:OtherAmalgams} have finite and additive Morley rank with DMP. 
In particular, the generic automorphism is axiomatisable in these collapsed theories using geometric axioms as in \ref{Ex:AddDMP}.\qed
\end{Thm}

\end{document}